\documentclass[10pt]{article}
\usepackage{amsmath, amsthm, amssymb, amsfonts, mathrsfs, mathtools}
\usepackage{graphicx}
\usepackage{makeidx}
\usepackage[left=2cm,right=2cm,top=2cm,bottom=2cm]{geometry}
\usepackage{caption}
\usepackage{subcaption}
\usepackage[section]{placeins}
\usepackage{enumitem}
\usepackage{tikz}
\usepackage{stmaryrd}

\newcommand{\Mod}[1]{\ (\mathrm{mod}\ #1)}

\usetikzlibrary{decorations.pathreplacing}
\tikzstyle{vertex}=[auto=left,circle,draw=black,fill=white, inner sep=1.5]

\usepackage{hyperref}
\hypersetup{colorlinks=true, citecolor={blue}, linkcolor=blue, filecolor=magenta, urlcolor=cyan}

\newtheorem{theorem}{Theorem}[section]

\newtheorem{lema}[theorem]{Lemma}
\newtheorem{corollary}{Corollary}[theorem]

\newtheorem{ex}{Example}[section]

\renewcommand{\i}{\mathbf{i}}

\newcommand{\Cay}{\operatorname{Cay}}

\newcommand{\ord}{\operatorname{ord}}

\newcommand{\Tr}{\operatorname{Tr}}
\newcommand{\Gal}{\operatorname{Gal}}
\newcommand{\IRR}{\operatorname{IRR}}
\newcommand{\Irr}{\operatorname{Irr}}

\title{H-integral and Gaussian integral normal mixed Cayley graphs}

\author{Monu Kadyan and Bikash Bhattacharjya\\
Department of Mathematics\\
Indian Institute of Technology Guwahati, India\\
monu.kadyan@iitg.ac.in\\ b.bikash@iitg.ac.in}

\date{}

\begin{document}
	\maketitle
	
	\vspace{-0.3in}
	
\begin{center}{\textbf{Abstract}}\end{center}
	
\noindent If all the eigenvalues of the Hermitian-adjacency matrix of a mixed graph are integers, then the mixed graph is called \emph{H-integral}. If all the eigenvalues of the (0,1)-adjacency matrix of a mixed graph are \emph{Gaussian integers}, then the mixed graph is called \emph{Gaussian integral}. For any finite group $\Gamma$, we characterize the set $S$ for which the normal mixed Cayley graph $\text{Cay}(\Gamma, S)$ is H-integral. We further prove that a normal mixed Cayley graph is H-integral if and only if it is Gaussian integral.


\vspace*{0.3cm}

\section{Introduction}

A \emph{mixed graph} $G$ is a pair $(V(G),E(G))$, where $V(G)$ and $E(G)$ are the vertex set and the edge set of $G$, respectively. In this case, $E(G)\subseteq V(G) \times V(G)\setminus \{(u,u)\colon u\in V(G)\}$. If $G$ is a mixed graph, then $(u,v)\in E(G)$ need not imply that $(v,u)\in E(G)$;  for further information see \cite{2015mixed}. If both $(u,v)$ and $(v,u)$ are members of $E(G)$, then we call $(u,v)$ to be an \textit{undirected} edge. If only one of $(u,v)$ and $(v,u)$ is a member of $E(G)$, then we call $(u,v)$ to be a \textit{directed} edge. It is clear that a mixed graph $G$ can have both undirected and directed edges. If all the edges of $G$ are undirected (resp. directed) then we call $G$ to be a \textit{simple graph} (resp. \textit{oriented graph}). Some definitions and results of this paper have similarities with those in the paper~\cite{kadyanHS-IntegSecKind}. Throughout the paper, we consider $\mathbf{i}=\sqrt{-1}$.

The (0,1)-\textit{adjacency matrix} and the \textit{Hermitian-adjacency matrix} of a mixed graph $G$ on $n$ vertices are denoted by $\mathcal{A}(G)=(a_{uv})_{n\times n}$ and $\mathcal{H}(G)=(h_{uv})_{n\times n}$, respectively, where 
\[a_{uv} = \left\{ \begin{array}{rl}
	1 &\mbox{ if }
	(u,v)\in E \\ 
	0 &\textnormal{ otherwise,}
\end{array}\right.     ~~~~~\text{ and }~~~~~~ h_{uv} = \left\{ \begin{array}{rl}
	1 &\mbox{ if }
	(u,v)\in E \textnormal{ and } (v,u)\in E \\ \mathbf{i} & \mbox{ if } (u,v)\in E \textnormal{ and } (v,u)\not\in E \\
	-\mathbf{i} & \mbox{ if } (u,v)\not\in E \textnormal{ and } (v,u)\in E\\
	0 &\textnormal{ otherwise.}
\end{array}\right.\] 

The Hermitian-adjacency matrix of mixed graphs was introduced by Liu and Li \cite{2015mixed} in 2015, and later by Guo and Mohar \cite{2017mixed} independently. Indeed,  Bapat et al. \cite{bapat2012weighted} introduced the notion of 3-colored digraph and its adjacency matrix in 2012. The Hermitian-adjacency matrix of mixed graphs is a special case of the adjacency matrix of a 3-colored digraph.

Let $G$ be a mixed graph. We refer to an eigenvalue of $\mathcal{H}(G)$ as an \emph{H-eigenvalue} of $G$. An eigenvalue of $\mathcal{A}(G)$ is referred to as an \emph{eigenvalue} of $G$. Similarly, the \emph{H-spectrum} of $G$ is the multi-set of the H-eigenvalues of $G$, and the \emph{spectrum} of $G$ is the multi-set of the eigenvalues of $G$. Since $\mathcal{H}(G)$  is a Hermitian matrix,  its H-eigenvalues are real numbers. However, if $G$ has at least one directed edge then $\mathcal{A}(G)$ is not symmetric, and so the eigenvalues of $G$ may or may not be real numbers.

If all of the H-eigenvalues of a mixed graph $G$ are integers, it is said to be \textit{H-integral}. If all of the eigenvalues of a mixed graph $G$ are Gaussian integers, then the mixed graph is said to be \textit{Gaussian integral}. The term \emph{integral graph} refers to an H-integral simple graph. For a simple graph $G$, note that $\mathcal{A}(G)=\mathcal{H}(G)$. As a result, the terms H-eigenvalue, H-spectrum and H-integrality of a simple graph $G$ have the same meaning with that of the eigenvalue, spectrum and integrality of $G$, respectively.

In 1974, Harary and Schwenk~\cite{harary1974graphs} proposed the question of characterization of integral  graphs. This problem has inspired a lot of interest over the last half-century. For more information on integral graphs, we refer the reader to \cite{ahmadi2009graphs, balinska2002survey, csikvari2010integral, watanabe1979note, watanabe1979integral}.


Throughout the paper, we consider $\Gamma$ to be a finite group with identity element ${\mathbf 1}$. Let $S$ be a subset of $\Gamma$ that does not contain ${\mathbf 1}$.  If $S$ is closed under inverse (resp. $a^{-1} \not\in S$ for all $a\in S$), it is said to be \textit{symmetric} (resp. \textit{skew-symmetric}). Define $\overline{S}= \{u\in S\colon u^{-1}\not\in S \}$. Clearly, $S\setminus \overline{S}$ is symmetric, while $\overline{S}$ is skew-symmetric.  The \textit{mixed Cayley graph} ${\rm Cay}(\Gamma,S)$ is a mixed graph, where $V({\rm Cay}(\Gamma,S))=\Gamma$ and $E({\rm Cay}(\Gamma,S))=\{ (a,b)\colon a,b\in \Gamma, ba^{-1}\in S \}$. If $S$ is symmetric (resp. skew-symmetric), we call ${\rm Cay}(\Gamma,S)$ to be a \textit{simple Cayley graph} (resp. \textit{oriented Cayley graph}). A mixed Cayley graph ${\rm Cay}(\Gamma, S)$ is called \textit{normal} if $S$ is the union of some conjugacy classes of the group $\Gamma$.

In 1982, Bridge and Mena \cite{bridges1982rational} presented a characterization of integral Cayley graphs over abelian groups. Later on, same characterization was obtained by Wasin So \cite{2006integral} for cyclic groups in 2005. In 2009, Klotz and Sander \cite{klotz2010integral} proved that if ${\rm Cay}(\Gamma,S)$ over an abelian group $\Gamma$ is integral, then $S$ belongs to the Boolean algebra $\mathbb{B}(\Gamma)$ generated by the subgroups of $\Gamma$. Moreover, they conjectured that the converse is also true, which was proved by Alperin and Peterson \cite{alperin2012integral}. For results on integral Cayley graphs over non-abelian groups, we refer the reader to \cite{cheng2019integral,ku2015cayley, lu2018integral}. In \cite{kadyan2021integral} and \cite{kadyan2021integralAbelian}, we characterized H-integral mixed Cayley graphs over cyclic group and abelian group in terms of their connection set. In 2014, Godsil \emph{et al.} \cite{godsil2014rationality} characterized integral normal Cayley graphs. Xu \emph{et al.} \cite{xu2011gaussian} and Li \cite{li2013circulant} characterized the set $S$ for which the mixed circulant graph ${\rm Cay}(\mathbb{Z}_n, S)$ is Gaussian integral.

The paper is organized as follows. In Section~\ref{Prlemch4}, we present some preliminary notions and known results. We also express the H-eigenvalues of a normal mixed Cayley graph $\Cay(\Gamma, S)$ in terms of the irreducible characters of $\Gamma$. In Section~\ref{H-integral}, we characterize the set $S$ for which the normal mixed Cayley graph ${\rm Cay}(\Gamma, S)$ is H-integral. This extends the results of \cite{kadyan2021integral,kadyan2021integralAbelian} to normal mixed Cayley graphs. In Section~\ref{gussian-integral},  we  prove that a normal mixed Cayley graph is H-integral if and only if it is \emph{Gaussian integral}.


\section{Preliminaries}\label{Prlemch4}

In this section, we determine the H-eigenvalues of a normal mixed Cayley graph $\Cay(\Gamma, S)$ in terms of the irreducible characters of $\Gamma$. Finally, we show that the normal mixed Cayley graph is H-integral if and only if each of its directed and undirected portions are H-integral.

For $x\in \Gamma$, let $\ord(x)$ denote the order of $x$. If $g$ and $h$ are elements of the group $\Gamma$, then we call $h$ a \textit{conjugate}\index{conjugate} of $g$ if $g=x^{-1}hx$ for some $x\in \Gamma$. The \textit{conjugacy class}\index{conjugacy class} of $g$, denoted ${\rm Cl}(g)$, is the set of all conjugates of $g$ in $\Gamma$. Define $C_{\Gamma}(g)$ to be the set of all elements of $\Gamma$ that commute with $g$. We denote the \textit{group algebra}\index{group algebra} of $\Gamma$ over a field $\mathbb{F}$ by $\mathbb{F}\Gamma$. That is, $\mathbb{F}\Gamma$ is the set of all formal sums $\sum\limits_{g\in \Gamma}a_g g$, where $a_g\in \mathbb{F}$, and we assume $1.g=g$ to have $\Gamma \subseteq \mathbb{F}\Gamma$.

A \textit{representation}\index{representation} of a finite group $\Gamma$ is a homomorphism $\rho \colon \Gamma \to \text{GL}_n(\mathbb{C})$, where $\text{GL}_n(\mathbb{C})$ is  the set of all $n\times n$ invertible matrices with complex entries. Here, the number $n$ is called the \textit{degree}\index{degree} of $\rho$. Two representations $\rho_1$ and $\rho_2$ of $\Gamma$ of degree $n$ are \textit{equivalent}\index{equivalent} if there is a $T\in \text{GL}_n(\mathbb{C})$ such that $T\rho_1(x)=\rho_2(x)T$ for each $x\in \Gamma$.

Let $\rho \colon \Gamma \to \text{GL}_n(\mathbb{C})$ be a representation of $\Gamma$. The \textit{character}\index{character} $\chi_{\rho}\colon \Gamma \to \mathbb{C}$ of $\rho$ is defined by setting $\chi_{\rho}(x):=\Tr(\rho(x))$ for $x\in \Gamma$, where $\Tr(\rho(x))$ is the trace of $\rho(x)$. By degree of $\chi_{\rho}$, we mean the degree of $\rho$, which is simply $\chi_{\rho}(\textbf{1})$. If $W$ is a $\rho(x)$-invariant subspace of $\mathbb{C}^n$ for each $x\in \Gamma$, then we say that $W$ is a $\rho(\Gamma)$-invariant subspace of $\mathbb{C}^n$. If $\{ \mathbf{0}\}$ and $\mathbb{C}^n$ are the only $\rho(\Gamma)$-invariant subspaces of $\mathbb{C}^n$, then  we say $\rho$ an \textit{irreducible representation}\index{irreducible representation} of $\Gamma$, and the corresponding character $\chi_{\rho}$ an \textit{irreducible character}\index{irreducible character} of $\Gamma$. 

For a group $\Gamma$, we denote by $\IRR(\Gamma)$ and $\Irr(\Gamma)$ the complete set of non-equivalent irreducible representations of $\Gamma$  and the complete set of non-equivalent irreducible characters of $\Gamma$, respectively. For $z\in \mathbb{C}$, let $\overline{z}$ denote the complex conjugate of $z$ and $\Re (z)$ (resp. $\Im (z)$) denote the real part (resp. imaginary part) of the complex number $z$.

\begin{theorem}[\cite{steinberg2009representation}]\label{NewThmChap1Added}
Let $\Gamma$ be a finite group and $\rho$ be a representation of $\Gamma$ of degree $k$ with corresponding  character $\chi$. If $x\in \Gamma$ and $\ord(x)=m$, then the following assertions hold.
\begin{enumerate}[label=(\roman*)]
\item $\rho(x)$ is similar to a diagonal matrix with diagonal entries $\epsilon_1,\hdots,\epsilon_k$, where $\epsilon_i^m=1$ for each \linebreak[4] $i\in \{1,\hdots,k\}$.
\item $\chi(x)= \sum\limits_{i=1}^{k}\epsilon_i$, where $\epsilon_i^m=1$ for each $i\in \{1,\hdots,k\}$.
\item $\chi(x^{-1})=\overline{\chi(x)}$.
\end{enumerate}
\end{theorem} 
\begin{proof} Note that $\rho(x)^m$ is an identity matrix. Therefore, $\rho(x)$ is diagonalizable, and that its eigenvalues  are $m$-th roots of unity. Thus the proofs of  Part (i) and Part (ii) follow. 

Again, $xx^{-1}=\mathbf{1}$ gives that $\rho(x^{-1})=\rho(x)^{-1}$. Therefore if $\chi(x)= \sum_{i=1}^{k}\epsilon_i$, then we have that $\chi(x^{-1})= \sum_{i=1}^{k}\epsilon_i^{-1}= \sum_{i=1}^{k}\overline{\epsilon}_i= \overline{\chi(x)}$.
\end{proof}

For a representation $\rho\colon \Gamma \to {\rm GL}_n(\mathbb{C})$ of $\Gamma$, define $\overline{\rho}\colon \Gamma \to {\rm GL}_n(\mathbb{C})$ by $\overline{\rho}(x):=\overline{\rho(x)}$, where $\overline{\rho(x)}$ is the matrix whose entries are the complex conjugates of the corresponding entries of $\rho(x)$. Note that if $\rho$ is irreducible, then $\overline{\rho}$ is also irreducible. Hence we have the following lemma. See Proposition 9.1.1 and Corollary 9.1.2 in \cite{steinberg2009representation} for details. 

\begin{lema}[\cite{steinberg2009representation}]\label{newLemmaConjugateChara}
	Let $\Gamma$ be a finite group and ${\rm Irr}(\Gamma)=\{ \chi_{1},\ldots,\chi_h \}$. If $j\in \{ 1,\ldots,h\}$, then there exists $k\in \{ 1,\ldots,h\}$ satisfying $\overline{\chi}_k=\chi_j$, where $\overline{\chi}_k\colon \Gamma \to \mathbb{C}$ such that $\overline{\chi}_k(x)=\overline{\chi_k(x)}$ for each $x \in \Gamma$.
\end{lema}

\begin{theorem}[\cite{steinberg2009representation}]\label{OrthoCharaRepr}
Let $\Gamma$ be a finite group and $x,y \in \Gamma$. If $\Irr(\Gamma)=\{ \chi_1,\ldots,\chi_h \}$, then 
\begin{enumerate}[label=(\roman*)]
\item 
\begin{align}
	\sum_{x\in \Gamma} \chi_{j}(x) \overline{\chi_{k} (x)} = \left\{ \begin{array}{cl}
		|\Gamma| & \mbox{ if $j=k$} \\
		0 &   \mbox{ otherwise}, 
	\end{array}\right.\nonumber
\end{align}
\item 
\begin{align}
	\sum_{j=1}^{h} \chi_{j}(x) \overline{\chi_{j} (y)} = \left\{ \begin{array}{ll}
		|C_{\Gamma}(x)| & \mbox{ if $x$ and $y$ are conjugates to each other} \\
		0 &   \mbox{ otherwise}. 
	\end{array}\right.\nonumber
	\end{align}
\end{enumerate}
\end{theorem}

For a function $f \colon \Gamma \to \mathbb{C}$, let $[f(yx^{-1})]_{x,y \in \Gamma}$ be the matrix whose rows and columns are indexed by the elements of $\Gamma$, and for $x,y \in \Gamma$, the $(x,y)$-th entry of the matrix is $f(yx^{-1})$. 

\begin{theorem}[\cite{foster2016spectra}]\label{EigNorColCayMix1}
Let $\Gamma$ be a finite group and $\Irr(\Gamma)=\{\chi_1,\ldots,\chi_h \}$. If $f \colon \Gamma \to \mathbb{C}$ is a class function, then the spectrum of the matrix $[f(yx^{-1})]_{x,y \in \Gamma}$ is $\{[\gamma_1]^{d_1^2},\ldots,[\gamma_h]^{d_h^2}\}$, where $$\gamma_{j} = \frac{1}{\chi_j(\mathbf 1)} \sum_{x\in \Gamma} f(x)\chi_{j}(x) \hspace{0.2cm} \textnormal{ and } \hspace{0.2cm} d_j = \chi_j(\mathbf{1})$$ for each $j \in \{ 1,\ldots , h\}$.
\end{theorem} 

\begin{lema}\label{EigNorCayMix1}
Let $\Gamma$ be a finite group. If $\Irr(\Gamma)=\{ \chi_1,\ldots,\chi_h \}$, then the H-spectrum of the normal mixed Cayley graph $\Cay(\Gamma, S)$ is $\{ [\gamma_{1}]^{d_1^2},\ldots, [\gamma_{h}]^{d_h^2} \},$ where $\gamma_j=\lambda_j + \mu_j$, $$\lambda_j= \frac{1}{\chi_j(1)} \sum_{s \in S \setminus \overline{S}} \chi_j(s),\hspace{0.2cm} \mu_j=\frac{\i}{\chi_j(1)}  \sum_{s \in \overline{S}}( \chi_j(s)-\chi_j(s^{-1})),$$ $\textnormal{ and } d_j=\chi_j(\mathbf 1) \textnormal{ for each } j\in \{1,\ldots,h\}$.
\end{lema}
\begin{proof}
Let $f\colon \Gamma \to \{0,1,\i,-\i\}$ be the function such that 
	$$f(s)= \left\{ \begin{array}{rl}
		1 & \mbox{if } s\in S \setminus \overline{S} \\
		\i & \mbox{if } s\in \overline{S}\\
		-\i & \mbox{if } s\in \overline{S}^{-1}\\ 
		0 &   \mbox{otherwise}. 
	\end{array}\right.
	$$
Since $S$ is a union of some conjugacy classes of $\Gamma$,  $f$ is a class function. The Hermitian adjacency matrix of $\Cay(\Gamma, S)$ is equal to $[f(yx^{-1})]_{x,y \in \Gamma}$. By Theorem~\ref{EigNorColCayMix1}, $$\gamma_j= \frac{1}{\chi_j(1)} \bigg( \sum_{s \in S \setminus \overline{S}} \chi_j(s) + \sum_{s \in \overline{S}}\i \chi_j(s)+ \sum_{s \in \overline{S}^{-1}} (-\i)\chi_j(s) \bigg),$$ and the result follows.
\end{proof}

As special cases of Lemma~\ref{EigNorCayMix1}, we have the following two corollaries.

\begin{corollary}\label{coro1EigNorCayMix1}
Let $\Gamma$ be a finite group. If $\Irr(\Gamma)=\{ \chi_1,\ldots,\chi_h \}$, then the H-spectrum (or spectrum) of the normal simple Cayley graph $\Cay(\Gamma, S)$ is $\{ [\lambda_{1}]^{d_1^2},\ldots, [\lambda_{h}]^{d_h^2} \},$ where $$\lambda_j= \frac{1}{\chi_j(1)} \sum_{s \in S } \chi_j(s) \textnormal{ and } d_j=\chi_j(\mathbf 1) \hspace{0.2cm} \textnormal{ for each } j\in \{1,\ldots,h\}.$$
\end{corollary}

\begin{corollary}\label{coro2EigNorCayMix1}
Let $\Gamma$ be a finite group. If $\Irr(\Gamma)=\{ \chi_1,\ldots,\chi_h \}$, then the H-spectrum of the normal oriented Cayley graph $\Cay(\Gamma, S)$ is $\{ [\mu_{1}]^{d_1^2},\ldots, [\mu_{h}]^{d_h^2} \},$ where $$ \mu_j=\frac{\i}{\chi_j(1)}  \sum_{s \in S}( \chi_j(s)-\chi_j(s^{-1})) \textnormal{ and } d_j=\chi_j(\mathbf 1) \hspace{0.2cm} \textnormal{ for each } j\in \{1,\ldots,h\}.$$
\end{corollary}

\begin{lema}\label{SeprateInteg}
If $\Gamma$ is a finite group, then the normal mixed Cayley graph ${\rm Cay}(\Gamma, S)$ is H-integral if and only if ${\rm Cay}(\Gamma, S\setminus \overline{S})$ is integral (or H-integral) and ${\rm Cay}(\Gamma, \overline{S})$ is H-integral.
\end{lema}
\begin{proof}
Let ${\rm Irr}(\Gamma)=\{ \chi_1,\ldots,\chi_h \}$ and $\gamma_{j}$ be an H-eigenvalue of the normal mixed Cayley graph ${\rm Cay}(\Gamma, S)$. By Lemma~\ref{EigNorCayMix1}, we have $\gamma_{j}=\lambda_{j}+\mu_{j}$, where $$\lambda_j= \frac{1}{\chi_j({\mathbf 1})} \sum\limits_{s \in S \setminus \overline{S}} \chi_j(s) \textnormal{ and } \mu_j=\frac{{\mathbf i}}{\chi_j({\mathbf 1})}  \sum\limits_{s \in \overline{S}}( \chi_j(s)-\chi_j(s^{-1}))$$ for each $j\in \{1,\ldots,h\}$. Assume that ${\rm Cay}(\Gamma, S)$ is H-integral and $j\in \{ 1,\ldots,h\}$. By Lemma~\ref{newLemmaConjugateChara}, there exists $k\in \{ 1,\ldots,h\}$ such that $\overline{\chi}_k=\chi_j$. Thus 
$$\gamma_k=\overline{\gamma}_k=\overline{\lambda}_k+ \overline{\mu}_k=\lambda_j-\mu_j.$$ 

By assumption, $\gamma_j$ and $\gamma_k$ are integers. As $\gamma_j=\lambda_j+\mu_j$ and $\gamma_k=\lambda_j-\mu_j$, we get $\lambda_j=\frac{\gamma_j+\gamma_k}{2}$ and $\mu_j=\frac{\gamma_j-\gamma_k}{2}$. Thus $\lambda_j$ and $\mu_j$ are rational algebraic integers, and so they are integers. Hence by Corollaries \ref{coro1EigNorCayMix1} and \ref{coro2EigNorCayMix1}, ${\rm Cay}(\Gamma, S\setminus \overline{S})$ is integral and ${\rm Cay}(\Gamma, \overline{S})$ is H-integral.

Conversely, assume that ${\rm Cay}(\Gamma, S\setminus \overline{S})$ is integral and ${\rm Cay}(\Gamma,  \overline{S})$ is H-integral. Using Lemma~\ref{EigNorCayMix1}, ${\rm Cay}(\Gamma, S)$ is H-integral.
\end{proof}

Let $n\geq 2$ be a positive integer.  For a divisor $d$ of $n$, define $G_n(d)=\{k: 1\leq k\leq n-1, \gcd(k,n)=d \}$. It is clear that $G_n(d)=dG_{\frac{n}{d}}(1)$.

Let $\mathbb{B}({\Gamma})$ be the boolean algebra\index{boolean algebra} generated by the subgroups of $\Gamma$. That is, $\mathbb{B}({\Gamma})$ is the set whose elements are obtained by intersections, unions and complements of subgroups of $\Gamma$. Define an equivalence relation $\sim$ on $\Gamma$ such that $x\sim y$ if and only if $y=x^k$ for some $k\in G_m(1)$, where $m=\ord(x)$. For $x\in \Gamma$, let $[ x ]$ denote the equivalence class of $x$ with respect to the relation $\sim$. Note that minimal non-empty sets in a boolean algebra are called its \textit{atoms}\index{atoms}. 

\begin{theorem}[\cite{alperin2012integral}]\label{NewThmAtomsAndEquiv}
The atoms of the boolean algebra $\mathbb{B}(\Gamma)$ are the sets $[x]$ for each $x\in \Gamma$.
\end{theorem}

By Theorem~\ref{NewThmAtomsAndEquiv}, we observe that each element of $\mathbb{B}(\Gamma)$ can be expressed as a disjoint union of the equivalence classes of the relation $\sim$ on $\Gamma$. Thus
\[\mathbb{B}(\Gamma)=\{[ x_1 ]\cup\cdots\cup [ x_k ]\colon  x_1,\ldots,x_k\in \Gamma, k\in \mathbb{N}\}.\]

\begin{theorem}[\cite{godsil2014rationality}]\label{NorMixCayGraphInteg}
Let $\Gamma$ be a finite group and ${\rm Cay}(\Gamma, S)$ be a normal simple Cayley graph. Then  ${\rm Cay}(\Gamma, S)$ is integral if and only if $S \in \mathbb{B}(\Gamma)$.
\end{theorem}

\noindent Let $n\equiv 0 \pmod 4$. For a divisor $d$ of $\frac{n}{4}$ and $r \in \{ 1,3\}$, define 
$$G_n^r(d)=\{ dk\colon k\equiv r \Mod 4, \gcd(dk,n )= d \}.$$ 
It is easy to see that $G_n(d)=G_n^1(d) \cup G_n^3(d), G_n^1(d) \cap G_n^3(d)=\emptyset$ and $G_n^r(d)=dG_{\frac{n}{d}}^r(1)$ for $r\in \{ 1,3\}$.

Let $\Gamma(4)$ be the set of all $x\in \Gamma$ satisfying $\ord(x)\equiv 0 \pmod 4$. That is, $\Gamma(4):=\{x\in \Gamma\colon \ord(x)\equiv 0 \pmod 4\}$. Define an equivalence relation $\approx$ on $\Gamma(4)$ such that $x \approx y$ if and only if $y=x^k$ for some  $k\in G_m^1(1)$, where $m=\ord(x)$. Observe that if  $x,y\in \Gamma(4)$ and $x \approx y$ then $x \sim y$, but the converse need not be true. For example, consider $x=5\pmod {12}$, $y=11\pmod {12}$ in $\mathbb{Z}_{12}$. Here $x,y\in \mathbb{Z}_{12}(4)$ and $x \sim y$, but $x \not\approx y$. For $x\in \Gamma(4)$, we denote the equivalence class of $x$ with respect to the relation $\approx$ by $\llbracket x \rrbracket$.  For $\Gamma(4) \neq \emptyset$, define $\mathbb{D}(\Gamma)$ to be the class of all skew-symmetric subsets $S$, where $S=\llbracket x_1 \rrbracket\cup\cdots\cup \llbracket x_k \rrbracket $ for some $x_1,\ldots,x_k\in \Gamma(4)$. For  $\Gamma(4) = \emptyset $, define $\mathbb{D}(\Gamma):=\{ \emptyset \}$. Thus
\[\mathbb{D}(\Gamma)=\left\{ \begin{array}{ll}
	\{\llbracket x_1 \rrbracket\cup\cdots\cup \llbracket x_k \rrbracket \colon x_1,\ldots,x_k\in \Gamma(4),  k\in \mathbb{N} \} &\mbox{ if }\Gamma(4)\neq \emptyset \\ 
	\{ \emptyset\} &\mbox{ if }\Gamma(4)= \emptyset.
\end{array}\right. \]

\section{H-integral normal mixed Cayley graphs}\label{H-integral}

Let the order of the group $\Gamma$ be $n$ and $\chi \in {\rm Irr}(\Gamma)$ be of degree $d$. Let $x\in \mathbb{Q}({\mathbf i})\Gamma$ be such that $x=\sum\limits_{g\in \Gamma} {\mathbf i} c_g g$, where $c_g\in \mathbb{Z}$ for all $g\in \Gamma$. Define $\chi(x) :=\sum\limits_{g\in \Gamma} {\mathbf i} c_g \chi(g)$. Note that $\chi(g)^n=\chi(g^n)=\chi({\mathbf 1})=d$, and so $\chi(g)$ is an algebraic integer for each $g\in \Gamma$. Therefore ${\mathbf i} c_g \chi(g)$ is an algebraic integer for each $g\in \Gamma$, and hence $\chi(x)$ is an algebraic integer. 

Let ${\rm Irr}(\Gamma)=\{ \chi_1,\ldots,\chi_h \}$. Let $E$ be the matrix $[E_{jg}]$ of size $h \times n$, whose rows are indexed by $1,\ldots,h$ and columns are indexed by the elements of $\Gamma$ such that $E_{jg}=\chi_{j}(g)$. Note that $EE^*=nI_h$ and the rank of $E$ is $h$, where $E^*$ is the conjugate transpose of $E$ and $I_h$ is the $h\times h$ identity matrix.  

Let ${\rm Gal}(\mathbb{K}/\mathbb{F})$ denote the Galois group of an extension $\mathbb{K}$ over the field $\mathbb{F}$. It is well known that ${\rm Gal}(\mathbb{Q}(\omega_m)/\mathbb{Q})=\{ \sigma_r \colon r\in G_m(1),\sigma_r(\omega_m)=\omega_m^r \}$. For example, see Section 14.5 in \cite{dummitandfoote}. If \linebreak[4] $m\equiv 0 \Mod 4$, then $\mathbb{Q}({\mathbf i},\omega_m)=\mathbb{Q}(\omega_m)$. Therefore, ${\rm Gal}(\mathbb{Q}({\mathbf i},\omega_m)/\mathbb{Q}({\mathbf i}))$ is a subgroup of ${\rm Gal}(\mathbb{Q}(\omega_m)/\mathbb{Q})$. Thus ${\rm Gal}(\mathbb{Q}({\mathbf i},\omega_m)/\mathbb{Q}({\mathbf i}))$ contains those automorphisms in ${\rm Gal}(\mathbb{Q}(\omega_m)/\mathbb{Q})$ that fix ${\mathbf i}$. Note that \linebreak[4] $G_m({1})=G_m^1(1)\cup G_m^3(1)$ and $G_m^1(1)\cap G_m^3(1)=\emptyset$. If $r\in G_m^1(1)$ then $\sigma_r({\mathbf i})={\mathbf i}$, and if $r \in G_m^3(1)$ then $\sigma_r({\mathbf i})=-{\mathbf i}$. Thus 
\[{\rm Gal}(\mathbb{Q}({\mathbf i},\omega_m)/\mathbb{Q}({\mathbf i}))= {\rm Gal}(\mathbb{Q}(\omega_m)/\mathbb{Q}({\mathbf i})) =\{ \sigma_r \colon r\in G_m^1(1), \sigma_r(\omega_m)=\omega_m^r\}.\]
If $m\not\equiv 0 \Mod 4$, then
$[\mathbb{Q}({\mathbf i},\omega_m) : \mathbb{Q}({\mathbf i})]= \varphi(m).$ 
Thus the field $\mathbb{Q}({\mathbf i},\omega_m)$ is a Galois extension of $\mathbb{Q}({\mathbf i})$ of degree $\varphi(m)$. Any automorphism of the field $\mathbb{Q}({\mathbf i},\omega_m)$ is uniquely determined by its action on $\omega_m$. Hence ${\rm Gal}(\mathbb{Q}({\mathbf i},\omega_m)/\mathbb{Q}({\mathbf i}))=\{ \tau_r \colon r\in G_m(1),\tau_r(\omega_m)=\omega_m^r \text{ and }\tau_r({\mathbf i})={\mathbf i}\}$. 

Let $g \in \Gamma$, $m=\ord(g)$, and $\chi$ be a character of $\Gamma$. By Theorem~\ref{NewThmChap1Added}, $\chi(g)= \sum_{i=1}^{k}\epsilon_i$, where $\epsilon_1,\ldots , \epsilon_k$ are some $m$-th roots of unity. If $m \equiv 0 \Mod 4$ and $\sigma_r \in \Gal(\mathbb{Q}(\i,\omega_m)/\mathbb{Q}(\i))$, then 
\begin{align*}
\sigma_r (\chi(g))  =  \sigma_r \left(\sum_{i=1}^{k}\epsilon_i \right)  =  \sum_{i=1}^{k}\sigma_r (\epsilon_i) =  \sum_{i=1}^{k}\epsilon_i^r = \chi(g^r).
\end{align*}
Similarly, if $m \not\equiv 0 \Mod 4$ and $\tau_r \in \Gal(\mathbb{Q}(\i,\omega_m)/\mathbb{Q}(\i))$, then also $\tau_r(\chi(g)) = \chi(g^r)$.

\begin{theorem}\label{MainTheoremIntChara}
Let $\Gamma$ be a finite group and ${\rm Irr}(\Gamma)=\{ \chi_1,\ldots,\chi_h \}$. If $x=\sum\limits_{g\in \Gamma} {\mathbf i} c_g g$, where $c_g\in \mathbb{Z}$ for all $g\in \Gamma$, then $\chi_j(x)$ is an integer for each $j\in \{1,\ldots, h\}$ if and only if the following conditions hold:
\begin{enumerate}[label=(\roman*)]
\item $\sum\limits_{s\in {\rm Cl}(g_1)}c_s=\sum\limits_{s\in {\rm Cl}(g_2)}c_s$ $\hspace{0.1cm}$ for each $g_1,g_2\in \Gamma(4)$ and $g_1\approx g_2$;
\item $\sum\limits_{s\in {\rm Cl}(g)}c_s=-\sum\limits_{s\in {\rm Cl}(g^{-1})}c_s$ $\hspace{0.1cm}$ for each $g \in \Gamma$;
\item $\sum\limits_{s\in {\rm Cl}(g)}c_s=0$ $\hspace{0.1cm}$ for all $g \in \Gamma \setminus \Gamma(4)$.
\end{enumerate}
\end{theorem}
\begin{proof}
Let $L$ be a set of representatives of the conjugacy classes in $\Gamma$. Since characters are class functions, we have 
\begin{equation}\label{NewEqCh4CgtComm}
\begin{split}
\chi_j(x)=\sum_{g\in L} \bigg( \sum_{s\in {\rm Cl}(g)} {\mathbf i} c_s\bigg) \chi_j(g) \textnormal{ for each } j\in \{1,\ldots, h\}.
\end{split}
\end{equation}
Assume that $\chi_j(x)$ is an integer for each $j\in \{1,\ldots, h\}$. 
Let $g_1,g_2\in \Gamma(4)$, $g_1 \approx g_2$ and $m=\ord(g_1)$. Therefore, there is $r\in G_m^1(1)$ and $\sigma_r \in {\rm Gal}(\mathbb{Q}(\omega_m)/\mathbb{Q}({\mathbf i}))$ such that $g_2=g_1^r$ and $\sigma_r(\omega_m)=\omega_m^r$. Note that $\sigma_r(\chi_j(g_1))= \chi_j(g_1^r)$ for each $j\in \{1,\ldots, h\}$. For $t\in \Gamma$, let $\theta_t=\sum\limits_{j=1}^{h} \chi_j(t) \overline{\chi}_j$, where $\overline{\chi}_j(g)=\overline{\chi_j(g)}$ for each $g\in \Gamma$. By Theorem~\ref{OrthoCharaRepr}, we have
\begin{align}
	\theta_t(u)= \left\{ \begin{array}{ll}
		|C_{\Gamma}(t)| & \mbox{ if $u$ and $t$ are conjugates to each other} \\
		0 &   \mbox{ otherwise}. 
	\end{array}\right.\nonumber
\end{align} 
So $\theta_t(x)= |C_{\Gamma}(t)| \sum\limits_{s\in {\rm Cl}(t)} {\mathbf i} c_s \in \mathbb{Q}(\mathbf i)$, and it gives that $\sigma_r(\theta_t(x))=\theta_t(x)$. Since $\chi_j(x)$ is assumed to be an integer, we have $\sigma_r(\chi_j(x))=\chi_j(x)$ for each $j\in \{1,\ldots, h\}$. Thus 
\begin{align}
|C_{\Gamma}(g_1)| \sum\limits_{s\in {\rm Cl}(g_1)}{\mathbf i}c_s=\theta_{g_1}(x)= \sigma_r(\theta_{g_1}(x)) &=\sum\limits_{j=1}^{h} \sigma_r(\chi_j(g_1)) \sigma_r(\overline{\chi}_j(x)) \nonumber\\
&=\sum\limits_{j=1}^{h} \chi_j(g_1^r) \overline{\chi}_j(x) \nonumber\\
&=\theta_{g_1^r}(x) =\theta_{g_2}(x)=|C_{\Gamma}(g_2)| \sum\limits_{s\in {\rm Cl}(g_2)}{\mathbf i}c_s. \label{ConjEqThetaEqua}
\end{align}
Since $g_1\approx g_2$, we have $C_{\Gamma}(g_1)=C_{\Gamma}(g_2)$. So Equation $(\ref{ConjEqThetaEqua})$ implies that $\sum\limits_{s\in {\rm Cl}(g_1)}c_s=\sum\limits_{s\in {\rm Cl}(g_2)}c_s$. Hence condition (i) holds. Again
\begin{align}
0=\chi_j(x)- \overline{\chi_j(x)}&=\sum_{g\in L} \bigg( \sum_{s\in {\rm Cl}(g)} {\mathbf i} c_s\bigg) \chi_j(g) - \sum_{g\in L} \bigg( \sum_{s\in {\rm Cl}(g)}- {\mathbf i} c_s\bigg) \overline{\chi_j(g)}\nonumber \\
&=\sum_{g\in L} \bigg( \sum_{s\in {\rm Cl}(g)} {\mathbf i} c_s\bigg) \chi_j(g) + \sum_{g\in L} \bigg( \sum_{s\in {\rm Cl}(g)} {\mathbf i} c_s\bigg) \chi_j(g^{-1})\nonumber\\
&=\sum_{g\in L} {\mathbf i} \bigg( \sum_{s\in {\rm Cl}(g)}  c_s + \sum_{s\in {\rm Cl}(g^{-1})}  c_s  \bigg) \chi_j(g),\nonumber
\end{align} and so 
\begin{align}
\sum_{g\in L} {\mathbf i} \bigg( \sum_{s\in {\rm Cl}(g)}  c_s + \sum_{s\in {\rm Cl}(g^{-1})}  c_s  \bigg) \begin{bmatrix}
			\chi_1(g) \\
						\vdots \\
			\chi_h(g)
		\end{bmatrix}=\begin{bmatrix}
			0 \\
			\vdots \\
			0
		\end{bmatrix}.\label{eqLIsumZero}
\end{align} 
Note that the number of irreducible characters of $\Gamma$ is equal to the number of conjugacy classes of $\Gamma$, that is, $|L|=h$. Since characters are class functions and the rank of $E$ is $h$, the columns of $E$ corresponding to the elements of $L$ are linearly independent. Thus by Equation $(\ref{eqLIsumZero})$, 
$$ \sum\limits_{s\in {\rm Cl}(g)}  c_s + \sum\limits_{s\in {\rm Cl}(g^{-1})}  c_s =0$$ for all $g\in L$, and so condition (ii) holds.

Let $g\in \Gamma \setminus \Gamma(4)$ and $m=\ord(g)$. Then there exists $\tau_{m-1} \in {\rm Gal}(\mathbb{Q}({\mathbf i},\omega_m)/\mathbb{Q}({\mathbf i}))$ such that \linebreak[4] $\tau_{m-1}(\omega_m)=\omega_m^{m-1}$. Note that $\tau_{m-1}(\chi_j(g))= \chi_j(g^{m-1})$ for each $j\in \{1,\ldots, h\}$. Now
\begin{align}
|C_{\Gamma}(g)| \sum\limits_{s\in {\rm Cl}(g)}{\mathbf i}c_s =\theta_{g}(x)
&= \tau_{m-1}(\theta_{g}(x))\nonumber\\
&=\sum\limits_{j=1}^{h} \tau_{m-1}(\chi_j(g)) \tau_{m-1}(\overline{\chi}_j(x)) \nonumber\\
&=\sum\limits_{j=1}^{h} \chi_j(g^{m-1}) \overline{\chi}_j(x) \nonumber\\
&=\theta_{g^{m-1}}(x)=\theta_{g^{-1}}(x)=|C_{\Gamma}(g^{-1})| \sum\limits_{s\in {\rm Cl}(g^{-1})}{\mathbf i}c_s. \label{ConjEqThetaEqua2}
\end{align} 
Since $C_{\Gamma}(g)=C_{\Gamma}(g^{-1})$, Equation $(\ref{ConjEqThetaEqua2})$ implies that $\sum\limits_{s\in {\rm Cl}(g)}c_s=\sum\limits_{s\in {\rm Cl}(g^{-1})}c_s$. This, together with condition (ii), gives $\sum\limits_{s\in {\rm Cl}(g)}c_s=0$ for all $g \in \Gamma \setminus \Gamma(4)$.  Hence condition (iii) also holds.

Conversely, assume that all the three conditions of the theorem hold. Let $n$ be the order of $\Gamma$. If $n\not\equiv 0\Mod 4$ then $\Gamma(4)=\emptyset$. Therefore by condition (iii) and Equation~(\ref{NewEqCh4CgtComm}), we have $\chi_j(x)=0$. Thus, $\chi_j(x)$ is an integer for each $j\in \{1,\ldots, h\}$. Now assume that $n\equiv 0 \Mod 4$. Let $L(4)$ be a set of representatives of the conjugacy classes of $\Gamma(4)$. Since characters are class functions, using condition (iii) we have 
\begin{align}
\chi_j(x)=\sum_{g\in L(4)} \bigg( \sum_{s\in {\rm Cl}(g)}{\mathbf i} c_s\bigg) \chi_j(g) \textnormal{ for each } j\in \{1,\ldots, h\}.\label{ConjEqThetaEqua5}
\end{align}
Let $\sigma_k \in {\rm Gal}(\mathbb{Q}({\mathbf i},\omega_n)/\mathbb{Q}({\mathbf i}))$. Therefore $\sigma_k(\omega_n)=\omega_n^k$ and $k\in G_n^1(1)$. Thus
\begin{align}
\sigma_k(\chi_j(x)) &= \sum_{g\in L(4)} \bigg( \sum_{s\in {\rm Cl}(g)} {\mathbf i} c_s\bigg) \sigma_k(\chi_j(g))\nonumber\\
&= \sum_{g\in L(4)} \bigg( \sum_{s\in {\rm Cl}(g)} {\mathbf i} c_s\bigg) \chi_j(g^k).\label{ConjEqThetaEqua3}
\end{align} Since $g\approx g^k$, by condition (i) we have $\sum\limits_{s\in {\rm Cl}(g)}c_s=\sum\limits_{s\in {\rm Cl}(g^k)}c_s$. From Equation $(\ref{ConjEqThetaEqua3})$, we get 
\begin{align}
\sigma_k(\chi_j(x)) &= \sum_{g\in L(4)} \bigg( \sum_{s\in {\rm Cl}(g^k)} {\mathbf i} c_s\bigg) \chi_j(g^k)=\chi_j(x). \label{neweq}
\end{align} 
The second equality in Equation $(\ref{neweq})$ holds, because $\{ g^k\colon g \in L(4)\}$ is also a set of representatives of conjugacy classes of $\Gamma(4)$. Now $\sigma_k(\chi_j(x)) = \chi_j(x)$ for each $k\in G_n^1(1)$, and so $\chi_j(x)\in \mathbb{Q}({\mathbf i})$. Taking complex conjugates in Equation $(\ref{ConjEqThetaEqua5})$, we have
\begin{align}
\overline{\chi_j(x)}=\sum_{g\in L(4)} \bigg( \sum_{s\in {\rm Cl}(g)} -{\mathbf i} c_s\bigg) \overline{\chi_j(g)} &= \sum_{g\in L(4)} \bigg( \sum_{s\in {\rm Cl}(g)} -{\mathbf i} c_s\bigg) \chi_j(g^{-1}) \nonumber\\
&= \sum_{g\in L(4)} \bigg( \sum_{s\in {\rm Cl}(g^{-1})} {\mathbf i} c_s\bigg) \chi_j(g^{-1})\nonumber\\
&= \chi_j(x). \label{ConjEqThetaEqua4}
\end{align}
Thus Equation $(\ref{ConjEqThetaEqua4})$ implies that $\chi_j(x)\in \mathbb{Q}$. As $\chi_j(x)$ is a rational algebraic integer, it must be an integer for each $j\in \{1,\ldots, h\}$.
\end{proof}

Indeed, we can replace condition (i) of Theorem \ref{MainTheoremIntChara} by $\sum\limits_{s\in {\rm Cl}(x)}c_s=\sum\limits_{s\in {\rm Cl}(y)}c_s$ for all $x,y \in \llbracket g \rrbracket$ and $g \in \Gamma(4)$.

\begin{theorem}\label{NorOriCayGraphInteg}
Let $\Gamma$ be a finite group and ${\rm Cay}(\Gamma, S)$ be a normal oriented Cayley graph. Then ${\rm Cay}(\Gamma, S)$ is H-integral if and only if $S \in \mathbb{D}(\Gamma)$.
\end{theorem}
\begin{proof}
Let ${\rm Irr}(\Gamma)=\{ \chi_1,\ldots,\chi_h \}$ and $x=\sum\limits_{g\in \Gamma} {\mathbf i} c_g g$, where 
$$c_g= \left\{ \begin{array}{rl}
		1 & \mbox{if } g\in S  \\
		-1 & \mbox{if } g\in S^{-1}\\
		0 &   \mbox{otherwise}. 
	\end{array}\right.$$ 
Observe that $\chi_j(x)= \sum\limits_{s\in S} {\mathbf i} (\chi_j(s)-\chi_j(s^{-1}))$, and so $\frac{\chi_j(x)}{\chi_j({\mathbf 1})}$ is an H-eigenvalue of ${\rm Cay}(\Gamma, S)$. Assume that the normal oriented Cayley graph ${\rm Cay}(\Gamma, S)$ is H-integral.  Thus $\chi_j(x)$ is an integer for each $j\in \{1,\ldots, h\}$, and therefore all the three conditions of Theorem \ref{MainTheoremIntChara} are satisfied for $x$. By the third condition of Theorem \ref{MainTheoremIntChara}, we get $\sum\limits_{s\in {\rm Cl}(g)}c_s=0$ for all $g \in \Gamma\setminus \Gamma (4)$. Note that $S$ is a union of some conjugacy classes of $\Gamma$. Therefore, if $g\in S$ then ${\rm Cl}(g)\subseteq S$, and so by the definition of $c_g$ we get $\sum\limits_{s\in {\rm Cl}(g)}c_s=|{\rm Cl}(g)|\neq 0$. Thus $S \cap (\Gamma \setminus \Gamma(4))=\emptyset$, that is, $S\subseteq \Gamma(4)$. Again, let  $g_1 \in S$, $g_2 \in \Gamma(4)$ and $ g_1 \approx g_2$. By the first condition of Theorem \ref{MainTheoremIntChara}, we get $0 < \sum\limits_{s\in {\rm Cl}(g_1)}c_s=\sum\limits_{s\in {\rm Cl}(g_2)}c_s$, which implies that $g_2\in S$. Thus $g_1 \in S$ gives $\llbracket g_1 \rrbracket \subseteq S$. Hence $S\in \mathbb{D}(\Gamma)$. 
	
	Conversely, assume that $S\in \mathbb{D}(\Gamma)$. Let ${\rm Cay}(\Gamma, S)$ be a normal oriented Cayley graph, so that $S$ is a union of some conjugacy classes of $\Gamma$. Let $$S= \llbracket x_1 \rrbracket\cup\cdots\cup \llbracket x_r \rrbracket = {\rm Cl}(y_1) \cup \cdots \cup {\rm Cl}(y_k) \subseteq \Gamma(4)$$ for some $x_1,\ldots,x_r,y_1,\ldots,y_k\in \Gamma(4)$. 
	Therefore $$S^{-1}= \llbracket x_1^{-1} \rrbracket\cup\cdots\cup \llbracket x_r^{-1} \rrbracket = {\rm Cl}(y_1^{-1}) \cup \cdots \cup {\rm Cl}(y_k^{-1}) \subseteq \Gamma(4).$$ Now for $g_1,g_2\in \Gamma(4)$, if $g_1\approx g_2$ then ${\rm Cl}(g_1),{\rm Cl}(g_2) \subseteq S$ or ${\rm Cl}(g_1),{\rm Cl}(g_2) \subseteq S^{-1}$ or \linebreak[4] ${\rm Cl}(g_1),{\rm Cl}(g_2) \subseteq (S \cup S^{-1})^{c}$. Here $(S \cup S^{-1})^{c}$ is the complement of $S \cup S^{-1}$ in $\Gamma$. Note that \linebreak[4] $|{\rm Cl}(g_1)|=|{\rm Cl}(g_2)|$. For all the cases, using the definition of $c_g$, we have $$\sum_{s\in {\rm Cl}(g_1)}c_s=\sum_{s\in {\rm Cl}(g_2)}c_s.$$  
	Again, we see that ${\rm Cl}(g) \subseteq S$ if and only if ${\rm Cl}(g^{-1}) \subseteq S^{-1}$. Therefore $$\text{$\sum\limits_{s\in {\rm Cl}(g)}c_s=-\sum\limits_{s\in {\rm Cl}(g^{-1})}c_s=0$ \hspace{0.5cm} or, \hspace{0.5cm} $\sum\limits_{s\in {\rm Cl}(g)}c_s= \pm |{\rm Cl}(g)|=-\sum\limits_{s\in {\rm Cl}(g^{-1})}c_s$.}$$ Further, if $g \not\in \Gamma(4)$ then ${\rm Cl}(g) \cap (S \cup S^{-1}) = \emptyset$, and so $\sum\limits_{s\in {\rm Cl}(g)}c_s=0$. Thus the three conditions of Theorem \ref{MainTheoremIntChara} are satisfied, and therefore $\chi_j(x)$ is an integer for each $j\in \{1,\ldots, h\}$. Consequently, the H-eigenvalue $\mu_j := \frac{\chi_j(x)}{\chi_j({\mathbf 1})}$ of ${\rm Cay}(\Gamma, S)$ is a rational algebraic integer, and hence it an integer for each $j\in \{1,\ldots, h\}$.
\end{proof}

We give the following example to illustrate Theorem~\ref{NorOriCayGraphInteg}.

\begin{ex}\label{ex1} \normalfont Consider the group  $M_{16}:= \langle a,x~|~ a^8=x^2={\mathbf 1}, xax^{-1}=a^5  \rangle $, and let $S=\{ a,a^5, a^3x, a^7x\}$. The conjugacy classes of $M_{16}$ are $\{{\mathbf 1}\}$,$\{a^4\}$,$\{a^2\}$,$\{a^6\}$,$\{a, a^5\}$, $\{a^3, a^7\}$, $\{ax, a^5x\}$, $\{a^3x, a^7x\}$, $\{x, a^4x\}$ and $\{a^2x, a^6x\}$. The normal oriented Cayley graph ${\rm Cay}(M_{16}, S)$ is shown in Figure~\ref{aaa}. We see that $S=\llbracket a \rrbracket \cup \llbracket a^3x \rrbracket={\rm Cl}(a)\cup {\rm Cl}(a^3x)$. Thus  $S \in \mathbb{D}(M_{16})$, and hence ${\rm Cay}(M_{16}, S)$ is H-integral. We can also confirm this by finding its  H-eigenvalues. Using the GAP software, the character table of $M_{16}$ is obtained and given in Table~\ref{table1}, where ${\rm Irr}(M_{16})=\{ \chi_1,\ldots, \chi_{10} \}$. Further, using Corollary~\ref{coro2EigNorCayMix1}, the H-spectrum of  ${\rm Cay}(M_{16}, S)$ is obtained as $\{ [\mu_{j}]^{1} \colon 1 \leq j \leq 8\}\cup \{[\mu_{9}]^{4},[\mu_{10}]^{4}\},$ where $\mu_j=0$ for $j \not\in \{5,6\}$, $\mu_5=-8$ and $\mu_6=8$. Thus all the H-eigenvalues of ${\rm Cay}(M_{16}, S)$ are integers. 
\end{ex}

\begin{theorem}\label{NormCayIntegChara}
Let $\Gamma$ be a finite group and ${\rm Cay}(\Gamma, S)$ be a normal mixed Cayley graph. Then ${\rm Cay}(\Gamma, S)$ is H-integral if and only if $S\setminus \overline{S} \in \mathbb{B}(\Gamma)$ and $\overline{S} \in \mathbb{D}(\Gamma)$.
\end{theorem}
\begin{proof}
By Lemma \ref{SeprateInteg}, ${\rm Cay}(\Gamma, S)$ is H-integral if and only if ${\rm Cay}(\Gamma, S \setminus \overline{S})$ is integral and ${\rm Cay}(\Gamma, \overline{S})$ is H-integral. Now the proof follows from Theorem~\ref{NorMixCayGraphInteg} and Theorem~\ref{NorOriCayGraphInteg}.
\end{proof}

The following example uses Theorem~\ref{NormCayIntegChara} to check H-integrality of a normal mixed Cayley graph.

\begin{ex}\label{example2} \normalfont Consider the group $M_{16}$ of Example~\ref{ex1} and $S=\{ a,a^3,a^5,a^7, a^3x, a^7x\}$. The normal mixed Cayley graph ${\rm Cay}(M_{16}, S)$ is shown in Figure~\ref{bbb}. We have $$S=[ a ] \cup \llbracket a^3x \rrbracket={\rm Cl}(a)\cup {\rm Cl}(a^3) \cup {\rm Cl}(a^3x).$$ Therefore $S \setminus \overline{S} \in \mathbb{B}(M_{16})$ and $\overline{S} \in \mathbb{D}(M_{16})$. Further, using Lemma~\ref{EigNorCayMix1}, the H-spectrum of  ${\rm Cay}(M_{16}, S)$ is obtained as 
$\{ [\gamma_{j}]^{1} \colon 1 \leq j \leq 8\}\cup \{[\gamma_{9}]^{4},[\gamma_{10}]^{4}\},$ where $\gamma_{1}=\gamma_{3}=\gamma_{6}=\gamma_{7}=4$, $\gamma_{2}=\gamma_{4}=\gamma_{5}=\gamma_{8}=-4$ and $\gamma_9=\gamma_{10}=0$. Thus ${\rm Cay}(M_{16}, S)$ is H-integral. 
\end{ex}

\begin{table}[ht]
\centering
\resizebox{\textwidth}{!}{
\begin{tabular}{ c c c c c c c c c c c }
	\hline
	& $\{{\mathbf 1}\}$ & $\{a^4\}$ &$\{a^2\}$ & $\{a^6\}$ & $\{a, a^5\}$ &$\{a^3, a^7\}$ &$\{ax, a^5x\}$ & $\{a^3x, a^7x\}$ & $\{x, a^4x\}$& $\{a^2x, a^6x\}$ \\
	\hline
	$\chi_1$ &$1$ &$1$ &$1$ &$1$ &$1$ &$1$ &$1$ &$1$ &$1$ &$1$ \\
	$\chi_2$ &$1$ &$1$ &$1$ &$1$ &$-1~~$ &$-1~~$ &$-1~~$ &$-1~~$&$1$ &$1$\\
	$\chi_3$ &$1$ &$1$ &$1$ &$1$ &$1$ &$1$ &$-1~~$ &$-1~~$&$-1~~$ &$-1~~$\\
	$\chi_4$ &$1$ &$1$ &$1$ &$1$ &$-1~~$ &$-1~~$ &$1$ &$1$ &$-1~~$ &$-1~~$\\
	$\chi_5$ &$1$ &$1$ &$-1~~$ &$-1~~$ &${\mathbf i}$ &$-{\mathbf i}~~$ &$-{\mathbf i}~~$ &${\mathbf i}$ &$-1~~$ &$1$ \\
	$\chi_6$ &$1$ &$1$ &$-1~~$ &$-1~~$ &$-{\mathbf i}~~$ &${\mathbf i}$ &${\mathbf i}$ &$-{\mathbf i}~~$ &$-1~~$ &$1$ \\
	$\chi_7$ &$1$ &$1$ &$-1~~$ &$-1~~$ &${\mathbf i}$ &$-{\mathbf i}~~$ &${\mathbf i}$ &$-{\mathbf i}~~$ &$1$ &$-1~~$ \\
	$\chi_8$ &$1$ &$1$ &$-1~~$ &$-1~~$ &$-{\mathbf i}~~$ &${\mathbf i}$ &$-{\mathbf i}~~$ & ${\mathbf i}$ &$1$ &$-1~~$ \\
	$\chi_9$ &$2$ &$-2~~$ &$2{\mathbf i}~$ &$-2{\mathbf i}~~~$ &$0$ &$0$ &$0$ &$0$ &$0$ &$0$\\
	$\chi_{10}$ &$2$ & $-2~~$ &$-2{\mathbf i}~~~$ &$2{\mathbf i}~$ &$0$ &$0$ &$0$ &$0$ &$0$ &$0$\\
	\hline
\end{tabular}}
\caption{Character table of $M_{16}$}\label{table1}
\end{table}



\begin{figure}[h]
\centering
\hfill
\begin{subfigure}{0.45\textwidth}

\tikzset{every picture/.style={line width=0.75pt}} 

\begin{tikzpicture}[x=0.2pt,y=0.2pt,yscale=-1,xscale=1]

\draw  [color={rgb, 255:red, 0; green, 0; blue, 0 }  ,draw opacity=1 ][line width=0.75]  (487.6,100.86) .. controls (487.6,93.84) and (493.3,88.15) .. (500.32,88.15) .. controls (507.35,88.15) and (513.04,93.84) .. (513.04,100.86) .. controls (513.04,107.89) and (507.35,113.58) .. (500.32,113.58) .. controls (493.3,113.58) and (487.6,107.89) .. (487.6,100.86) -- cycle ;
\draw  [color={rgb, 255:red, 0; green, 0; blue, 0 }  ,draw opacity=1 ][line width=0.75]  (641.6,131.86) .. controls (641.6,124.84) and (647.3,119.15) .. (654.32,119.15) .. controls (661.35,119.15) and (667.04,124.84) .. (667.04,131.86) .. controls (667.04,138.89) and (661.35,144.58) .. (654.32,144.58) .. controls (647.3,144.58) and (641.6,138.89) .. (641.6,131.86) -- cycle ;
\draw  [color={rgb, 255:red, 0; green, 0; blue, 0 }  ,draw opacity=1 ][line width=0.75]  (770.6,217.86) .. controls (770.6,210.84) and (776.3,205.15) .. (783.32,205.15) .. controls (790.35,205.15) and (796.04,210.84) .. (796.04,217.86) .. controls (796.04,224.89) and (790.35,230.58) .. (783.32,230.58) .. controls (776.3,230.58) and (770.6,224.89) .. (770.6,217.86) -- cycle ;
\draw  [color={rgb, 255:red, 0; green, 0; blue, 0 }  ,draw opacity=1 ][line width=0.75]  (204.6,218.86) .. controls (204.6,211.84) and (210.3,206.15) .. (217.32,206.15) .. controls (224.35,206.15) and (230.04,211.84) .. (230.04,218.86) .. controls (230.04,225.89) and (224.35,231.58) .. (217.32,231.58) .. controls (210.3,231.58) and (204.6,225.89) .. (204.6,218.86) -- cycle ;
\draw  [color={rgb, 255:red, 0; green, 0; blue, 0 }  ,draw opacity=1 ][line width=0.75]  (119.6,346.86) .. controls (119.6,339.84) and (125.3,334.15) .. (132.32,334.15) .. controls (139.35,334.15) and (145.04,339.84) .. (145.04,346.86) .. controls (145.04,353.89) and (139.35,359.58) .. (132.32,359.58) .. controls (125.3,359.58) and (119.6,353.89) .. (119.6,346.86) -- cycle ;
\draw  [color={rgb, 255:red, 0; green, 0; blue, 0 }  ,draw opacity=1 ][line width=0.75]  (87.6,500.86) .. controls (87.6,493.84) and (93.3,488.15) .. (100.32,488.15) .. controls (107.35,488.15) and (113.04,493.84) .. (113.04,500.86) .. controls (113.04,507.89) and (107.35,513.58) .. (100.32,513.58) .. controls (93.3,513.58) and (87.6,507.89) .. (87.6,500.86) -- cycle ;
\draw  [color={rgb, 255:red, 0; green, 0; blue, 0 }  ,draw opacity=1 ][line width=0.75]  (205.6,782.86) .. controls (205.6,775.84) and (211.3,770.15) .. (218.32,770.15) .. controls (225.35,770.15) and (231.04,775.84) .. (231.04,782.86) .. controls (231.04,789.89) and (225.35,795.58) .. (218.32,795.58) .. controls (211.3,795.58) and (205.6,789.89) .. (205.6,782.86) -- cycle ;
\draw  [color={rgb, 255:red, 0; green, 0; blue, 0 }  ,draw opacity=1 ][line width=0.75]  (333.6,869.86) .. controls (333.6,862.84) and (339.3,857.15) .. (346.32,857.15) .. controls (353.35,857.15) and (359.04,862.84) .. (359.04,869.86) .. controls (359.04,876.89) and (353.35,882.58) .. (346.32,882.58) .. controls (339.3,882.58) and (333.6,876.89) .. (333.6,869.86) -- cycle ;
\draw  [color={rgb, 255:red, 0; green, 0; blue, 0 }  ,draw opacity=1 ][line width=0.75]  (488.6,899.86) .. controls (488.6,892.84) and (494.3,887.15) .. (501.32,887.15) .. controls (508.35,887.15) and (514.04,892.84) .. (514.04,899.86) .. controls (514.04,906.89) and (508.35,912.58) .. (501.32,912.58) .. controls (494.3,912.58) and (488.6,906.89) .. (488.6,899.86) -- cycle ;
\draw  [color={rgb, 255:red, 0; green, 0; blue, 0 }  ,draw opacity=1 ][line width=0.75]  (770.6,782.86) .. controls (770.6,775.84) and (776.3,770.15) .. (783.32,770.15) .. controls (790.35,770.15) and (796.04,775.84) .. (796.04,782.86) .. controls (796.04,789.89) and (790.35,795.58) .. (783.32,795.58) .. controls (776.3,795.58) and (770.6,789.89) .. (770.6,782.86) -- cycle ;
\draw  [color={rgb, 255:red, 0; green, 0; blue, 0 }  ,draw opacity=1 ][line width=0.75]  (857.6,653.86) .. controls (857.6,646.84) and (863.3,641.15) .. (870.32,641.15) .. controls (877.35,641.15) and (883.04,646.84) .. (883.04,653.86) .. controls (883.04,660.89) and (877.35,666.58) .. (870.32,666.58) .. controls (863.3,666.58) and (857.6,660.89) .. (857.6,653.86) -- cycle ;
\draw  [color={rgb, 255:red, 0; green, 0; blue, 0 }  ,draw opacity=1 ][line width=0.75]  (886.6,499.86) .. controls (886.6,492.84) and (892.3,487.15) .. (899.32,487.15) .. controls (906.35,487.15) and (912.04,492.84) .. (912.04,499.86) .. controls (912.04,506.89) and (906.35,512.58) .. (899.32,512.58) .. controls (892.3,512.58) and (886.6,506.89) .. (886.6,499.86) -- cycle ;
\draw  [color={rgb, 255:red, 0; green, 0; blue, 0 }  ,draw opacity=1 ][line width=0.75]  (335.6,130.86) .. controls (335.6,123.84) and (341.3,118.15) .. (348.32,118.15) .. controls (355.35,118.15) and (361.04,123.84) .. (361.04,130.86) .. controls (361.04,137.89) and (355.35,143.58) .. (348.32,143.58) .. controls (341.3,143.58) and (335.6,137.89) .. (335.6,130.86) -- cycle ;
\draw  [color={rgb, 255:red, 0; green, 0; blue, 0 }  ,draw opacity=1 ][line width=0.75]  (641.6,869.86) .. controls (641.6,862.84) and (647.3,857.15) .. (654.32,857.15) .. controls (661.35,857.15) and (667.04,862.84) .. (667.04,869.86) .. controls (667.04,876.89) and (661.35,882.58) .. (654.32,882.58) .. controls (647.3,882.58) and (641.6,876.89) .. (641.6,869.86) -- cycle ;
\draw  [color={rgb, 255:red, 0; green, 0; blue, 0 }  ,draw opacity=1 ][line width=0.75]  (119.6,653.86) .. controls (119.6,646.84) and (125.3,641.15) .. (132.32,641.15) .. controls (139.35,641.15) and (145.04,646.84) .. (145.04,653.86) .. controls (145.04,660.89) and (139.35,666.58) .. (132.32,666.58) .. controls (125.3,666.58) and (119.6,660.89) .. (119.6,653.86) -- cycle ;
\draw  [color={rgb, 255:red, 0; green, 0; blue, 0 }  ,draw opacity=1 ][line width=0.75]  (857.6,347.86) .. controls (857.6,340.84) and (863.3,335.15) .. (870.32,335.15) .. controls (877.35,335.15) and (883.04,340.84) .. (883.04,347.86) .. controls (883.04,354.89) and (877.35,360.58) .. (870.32,360.58) .. controls (863.3,360.58) and (857.6,354.89) .. (857.6,347.86) -- cycle ;
\draw    (509,111.17) -- (890.58,491.42) ;
\draw    (509,111.17) -- (862.58,645.42) ;
\draw    (509,111.17) -- (783.32,770.15) ;
\draw    (509,111.17) -- (654.32,857.15) ;
\draw    (660.58,143.42) -- (890.58,491.42) ;
\draw    (660.58,143.42) -- (862.58,645.42) ;
\draw    (660.58,143.42) -- (783.32,770.15) ;
\draw    (660.58,143.42) -- (654.32,857.15) ;
\draw    (783.32,230.58) -- (890.58,491.42) ;
\draw    (783.32,230.58) -- (862.58,645.42) ;
\draw    (783.32,230.58) -- (783.32,770.15) ;
\draw    (783.32,230.58) -- (654.32,857.15) ;
\draw    (864.58,359.42) -- (890.58,491.42) ;
\draw    (864.58,359.42) -- (862.58,645.42) ;
\draw    (864.58,359.42) -- (783.32,770.15) ;
\draw    (864.58,359.42) -- (654.32,857.15) ;
\draw    (496.71,888.52) -- (108.58,510.42) ;
\draw    (496.71,888.52) -- (138.47,357.38) ;
\draw    (496.71,888.52) -- (216.64,231.96) ;
\draw    (496.71,888.52) -- (344.88,143.84) ;
\draw    (344.85,857.6) -- (108.58,510.42) ;
\draw    (344.85,857.6) -- (138.47,357.38) ;
\draw    (344.85,857.6) -- (216.64,231.96) ;
\draw    (344.85,857.6) -- (344.88,143.84) ;
\draw    (221.35,771.5) -- (108.58,510.42) ;
\draw    (221.35,771.5) -- (138.47,357.38) ;
\draw    (221.35,771.5) -- (216.64,231.96) ;
\draw    (221.35,771.5) -- (344.88,143.84) ;
\draw    (138.97,643.38) -- (108.58,510.42) ;
\draw    (138.97,643.38) -- (138.47,357.38) ;
\draw    (138.97,643.38) -- (216.64,231.96) ;
\draw    (138.97,643.38) -- (344.88,143.84) ;
\draw    (890.07,508.46) -- (510.56,890.78) ;
\draw    (890.07,508.46) -- (356.5,863.07) ;
\draw    (890.07,508.46) -- (231.62,784.05) ;
\draw    (890.07,508.46) -- (144.38,655.22) ;
\draw    (858.12,660.11) -- (510.56,890.78) ;
\draw    (858.12,660.11) -- (356.5,863.07) ;
\draw    (858.12,660.11) -- (231.62,784.05) ;
\draw    (858.12,660.11) -- (144.38,655.22) ;
\draw    (771.18,783.01) -- (510.56,890.78) ;
\draw    (771.18,783.01) -- (356.5,863.07) ;
\draw    (771.18,783.01) -- (231.62,784.05) ;
\draw    (771.18,783.01) -- (144.38,655.22) ;
\draw    (642.51,864.52) -- (510.56,890.78) ;
\draw    (642.51,864.52) -- (356.5,863.07) ;
\draw    (642.51,864.52) -- (231.62,784.05) ;
\draw    (642.51,864.52) -- (144.38,655.22) ;
\draw    (111.88,493.69) -- (491.27,111.24) ;
\draw    (111.88,493.69) -- (645.33,138.89) ;
\draw    (111.88,493.69) -- (770.24,217.87) ;
\draw    (111.88,493.69) -- (857.53,346.67) ;
\draw    (143.79,342.03) -- (491.27,111.24) ;
\draw    (143.79,342.03) -- (645.33,138.89) ;
\draw    (143.79,342.03) -- (770.24,217.87) ;
\draw    (143.79,342.03) -- (857.53,346.67) ;
\draw    (230.68,219.1) -- (491.27,111.24) ;
\draw    (230.68,219.1) -- (645.33,138.89) ;
\draw    (230.68,219.1) -- (770.24,217.87) ;
\draw    (230.68,219.1) -- (857.53,346.67) ;
\draw    (359.33,137.54) -- (491.27,111.24) ;
\draw    (359.33,137.54) -- (645.33,138.89) ;
\draw    (359.33,137.54) -- (770.24,217.87) ;
\draw    (359.33,137.54) -- (857.53,346.67) ;
\draw  [fill={rgb, 255:red, 0; green, 0; blue, 0 }  ,fill opacity=1 ] (539.01,239.18) -- (537.41,255.29) -- (529.92,240.94) -- (535.93,247.68) -- cycle ;
\draw  [fill={rgb, 255:red, 0; green, 0; blue, 0 }  ,fill opacity=1 ] (570.59,195.65) -- (575.16,211.19) -- (562.83,200.68) -- (570.93,204.68) -- cycle ;
\draw  [fill={rgb, 255:red, 0; green, 0; blue, 0 }  ,fill opacity=1 ] (563.3,228.75) -- (564.85,244.87) -- (554.73,232.23) -- (561.93,237.68) -- cycle ;
\draw  [fill={rgb, 255:red, 0; green, 0; blue, 0 }  ,fill opacity=1 ] (663.55,281.91) -- (658.95,297.44) -- (654.29,281.92) -- (658.93,289.68) -- cycle ;
\draw  [fill={rgb, 255:red, 0; green, 0; blue, 0 }  ,fill opacity=1 ] (788.67,406.98) -- (783.83,422.43) -- (779.41,406.85) -- (783.93,414.68) -- cycle ;
\draw  [fill={rgb, 255:red, 0; green, 0; blue, 0 }  ,fill opacity=1 ] (701.1,324.21) -- (699.32,340.31) -- (692,325.87) -- (697.93,332.68) -- cycle ;
\draw  [fill={rgb, 255:red, 0; green, 0; blue, 0 }  ,fill opacity=1 ] (599.59,194.95) -- (607.5,209.08) -- (593.14,201.59) -- (601.93,203.68) -- cycle ;
\draw  [fill={rgb, 255:red, 0; green, 0; blue, 0 }  ,fill opacity=1 ] (750.06,417.03) -- (742.34,431.27) -- (741,415.13) -- (743.93,423.68) -- cycle ;
\draw  [fill={rgb, 255:red, 0; green, 0; blue, 0 }  ,fill opacity=1 ] (765.5,537.56) -- (749.35,536.31) -- (763.55,528.52) -- (756.93,534.68) -- cycle ;
\draw  [fill={rgb, 255:red, 0; green, 0; blue, 0 }  ,fill opacity=1 ] (818.44,388.35) -- (816.01,404.36) -- (809.28,389.63) -- (814.93,396.68) -- cycle ;
\draw  [fill={rgb, 255:red, 0; green, 0; blue, 0 }  ,fill opacity=1 ] (815.42,295.76) -- (816.76,311.9) -- (806.8,299.13) -- (813.93,304.68) -- cycle ;
\draw  [fill={rgb, 255:red, 0; green, 0; blue, 0 }  ,fill opacity=1 ] (744.59,261.65) -- (749.16,277.19) -- (736.83,266.68) -- (744.93,270.68) -- cycle ;
\draw  [fill={rgb, 255:red, 0; green, 0; blue, 0 }  ,fill opacity=1 ] (711.25,257.74) -- (712.9,273.85) -- (702.69,261.27) -- (709.93,266.68) -- cycle ;
\draw  [fill={rgb, 255:red, 0; green, 0; blue, 0 }  ,fill opacity=1 ] (805.85,600.18) -- (791.02,606.7) -- (799.86,593.13) -- (796.93,601.68) -- cycle ;
\draw  [fill={rgb, 255:red, 0; green, 0; blue, 0 }  ,fill opacity=1 ] (802.95,572.09) -- (787.54,577.07) -- (797.71,564.47) -- (793.93,572.68) -- cycle ;
\draw  [fill={rgb, 255:red, 0; green, 0; blue, 0 }  ,fill opacity=1 ] (868.75,486.03) -- (863.75,501.43) -- (859.5,485.81) -- (863.93,493.68) -- cycle ;
\draw  [fill={rgb, 255:red, 0; green, 0; blue, 0 }  ,fill opacity=1 ] (819.31,478.46) -- (808.78,490.77) -- (810.86,474.71) -- (811.93,483.68) -- cycle ;
\draw  [fill={rgb, 255:red, 0; green, 0; blue, 0 }  ,fill opacity=1 ] (718.51,663.6) -- (703.18,658.38) -- (718.87,654.35) -- (710.93,658.68) -- cycle ;
\draw  [fill={rgb, 255:red, 0; green, 0; blue, 0 }  ,fill opacity=1 ] (766.93,564.5) -- (750.94,567.03) -- (762.93,556.15) -- (757.93,563.68) -- cycle ;
\draw  [fill={rgb, 255:red, 0; green, 0; blue, 0 }  ,fill opacity=1 ] (848.22,467.19) -- (840.15,481.23) -- (839.21,465.06) -- (841.93,473.68) -- cycle ;
\draw  [fill={rgb, 255:red, 0; green, 0; blue, 0 }  ,fill opacity=1 ] (880.85,416.12) -- (879.54,432.27) -- (871.8,418.04) -- (877.93,424.68) -- cycle ;
\draw  [fill={rgb, 255:red, 0; green, 0; blue, 0 }  ,fill opacity=1 ] (678.42,699.78) -- (662.31,698.12) -- (676.7,690.69) -- (669.93,696.68) -- cycle ;
\draw  [fill={rgb, 255:red, 0; green, 0; blue, 0 }  ,fill opacity=1 ] (742.9,711.78) -- (726.83,713.8) -- (739.17,703.31) -- (733.93,710.68) -- cycle ;
\draw  [fill={rgb, 255:red, 0; green, 0; blue, 0 }  ,fill opacity=1 ] (536.02,848.28) -- (521.24,841.66) -- (537.23,839.1) -- (528.93,842.68) -- cycle ;
\draw  [fill={rgb, 255:red, 0; green, 0; blue, 0 }  ,fill opacity=1 ] (708.93,813.5) -- (692.94,816.03) -- (704.93,805.15) -- (699.93,812.68) -- cycle ;
\draw  [fill={rgb, 255:red, 0; green, 0; blue, 0 }  ,fill opacity=1 ] (593.57,788.51) -- (578.18,783.46) -- (593.82,779.26) -- (585.93,783.68) -- cycle ;
\draw  [fill={rgb, 255:red, 0; green, 0; blue, 0 }  ,fill opacity=1 ] (708.39,799.87) -- (692.3,798.04) -- (706.76,790.76) -- (699.93,796.68) -- cycle ;
\draw  [fill={rgb, 255:red, 0; green, 0; blue, 0 }  ,fill opacity=1 ] (736.94,745.95) -- (721.61,751.17) -- (731.58,738.41) -- (727.93,746.68) -- cycle ;
\draw  [fill={rgb, 255:red, 0; green, 0; blue, 0 }  ,fill opacity=1 ] (585.81,749.54) -- (571.29,742.37) -- (587.36,740.42) -- (578.93,743.68) -- cycle ;
\draw  [fill={rgb, 255:red, 0; green, 0; blue, 0 }  ,fill opacity=1 ] (526.92,820.44) -- (513.54,811.32) -- (529.73,811.63) -- (520.93,813.68) -- cycle ;
\draw  [fill={rgb, 255:red, 0; green, 0; blue, 0 }  ,fill opacity=1 ] (159.48,638.7) -- (156.3,622.82) -- (167.66,634.37) -- (159.93,629.68) -- cycle ;
\draw  [fill={rgb, 255:red, 0; green, 0; blue, 0 }  ,fill opacity=1 ] (120.36,582.34) -- (121.02,566.15) -- (129.33,580.06) -- (122.93,573.68) -- cycle ;
\draw  [fill={rgb, 255:red, 0; green, 0; blue, 0 }  ,fill opacity=1 ] (155.12,534.59) -- (162.18,520.02) -- (164.25,536.08) -- (160.93,527.68) -- cycle ;
\draw  [fill={rgb, 255:red, 0; green, 0; blue, 0 }  ,fill opacity=1 ] (184.68,611.1) -- (186.63,595.03) -- (193.8,609.55) -- (187.93,602.68) -- cycle ;
\draw  [fill={rgb, 255:red, 0; green, 0; blue, 0 }  ,fill opacity=1 ] (134.59,515.6) -- (138.66,499.92) -- (143.84,515.26) -- (138.93,507.68) -- cycle ;
\draw  [fill={rgb, 255:red, 0; green, 0; blue, 0 }  ,fill opacity=1 ] (183.74,524.14) -- (193.84,511.48) -- (192.33,527.6) -- (190.93,518.68) -- cycle ;
\draw  [fill={rgb, 255:red, 0; green, 0; blue, 0 }  ,fill opacity=1 ] (523.68,868.33) -- (508.17,863.65) -- (523.71,859.08) -- (515.93,863.68) -- cycle ;
\draw  [fill={rgb, 255:red, 0; green, 0; blue, 0 }  ,fill opacity=1 ] (579.58,881.31) -- (563.4,880.54) -- (577.36,872.32) -- (570.93,878.68) -- cycle ;
\draw  [fill={rgb, 255:red, 0; green, 0; blue, 0 }  ,fill opacity=1 ] (275.93,702.65) -- (273.72,686.61) -- (284.36,698.83) -- (276.93,693.68) -- cycle ;
\draw  [fill={rgb, 255:red, 0; green, 0; blue, 0 }  ,fill opacity=1 ] (391.33,740.7) -- (386.67,725.19) -- (399.06,735.62) -- (390.93,731.68) -- cycle ;
\draw  [fill={rgb, 255:red, 0; green, 0; blue, 0 }  ,fill opacity=1 ] (302.88,678.18) -- (304.45,662.06) -- (311.96,676.4) -- (305.93,669.68) -- cycle ;
\draw  [fill={rgb, 255:red, 0; green, 0; blue, 0 }  ,fill opacity=1 ] (407.85,808.23) -- (399.02,794.65) -- (413.84,801.17) -- (404.93,799.68) -- cycle ;
\draw  [fill={rgb, 255:red, 0; green, 0; blue, 0 }  ,fill opacity=1 ] (340.35,717.46) -- (344.89,701.92) -- (349.6,717.41) -- (344.93,709.68) -- cycle ;
\draw  [fill={rgb, 255:red, 0; green, 0; blue, 0 }  ,fill opacity=1 ] (467.84,768.17) -- (469.48,752.05) -- (476.94,766.43) -- (470.93,759.68) -- cycle ;
\draw  [fill={rgb, 255:red, 0; green, 0; blue, 0 }  ,fill opacity=1 ] (257.52,737.69) -- (252.54,722.28) -- (265.14,732.45) -- (256.93,728.68) -- cycle ;
\draw  [fill={rgb, 255:red, 0; green, 0; blue, 0 }  ,fill opacity=1 ] (215.37,591.47) -- (219.88,575.92) -- (224.62,591.4) -- (219.93,583.68) -- cycle ;
\draw  [fill={rgb, 255:red, 0; green, 0; blue, 0 }  ,fill opacity=1 ] (439.31,766.56) -- (438.22,750.41) -- (447.98,763.33) -- (440.93,757.68) -- cycle ;
\draw  [fill={rgb, 255:red, 0; green, 0; blue, 0 }  ,fill opacity=1 ] (251.87,592.38) -- (259.46,578.07) -- (260.95,594.19) -- (257.93,585.68) -- cycle ;
\draw  [fill={rgb, 255:red, 0; green, 0; blue, 0 }  ,fill opacity=1 ] (322.44,301.61) -- (338.55,303.2) -- (324.2,310.69) -- (330.93,304.68) -- cycle ;
\draw  [fill={rgb, 255:red, 0; green, 0; blue, 0 }  ,fill opacity=1 ] (477.48,181.48) -- (490.13,191.59) -- (474.01,190.06) -- (482.93,188.68) -- cycle ;
\draw  [fill={rgb, 255:red, 0; green, 0; blue, 0 }  ,fill opacity=1 ] (413.6,251.24) -- (427.44,259.64) -- (411.26,260.19) -- (419.93,257.68) -- cycle ;
\draw  [fill={rgb, 255:red, 0; green, 0; blue, 0 }  ,fill opacity=1 ] (412.06,214.24) -- (427.69,218.48) -- (412.29,223.49) -- (419.93,218.68) -- cycle ;
\draw  [fill={rgb, 255:red, 0; green, 0; blue, 0 }  ,fill opacity=1 ] (263.91,257.18) -- (279.37,252.34) -- (269.09,264.85) -- (272.93,256.68) -- cycle ;
\draw  [fill={rgb, 255:red, 0; green, 0; blue, 0 }  ,fill opacity=1 ] (237.5,464.45) -- (253.58,466.35) -- (239.08,473.56) -- (245.93,467.68) -- cycle ;
\draw  [fill={rgb, 255:red, 0; green, 0; blue, 0 }  ,fill opacity=1 ] (276.59,338.41) -- (291.67,344.33) -- (275.81,347.63) -- (283.93,343.68) -- cycle ;
\draw  [fill={rgb, 255:red, 0; green, 0; blue, 0 }  ,fill opacity=1 ] (234,437.35) -- (250.11,435.72) -- (237.52,445.91) -- (242.93,438.68) -- cycle ;
\draw  [fill={rgb, 255:red, 0; green, 0; blue, 0 }  ,fill opacity=1 ] (194.13,404.71) -- (208.53,397.3) -- (200.54,411.39) -- (202.93,402.68) -- cycle ;
\draw  [fill={rgb, 255:red, 0; green, 0; blue, 0 }  ,fill opacity=1 ] (254.9,392.98) -- (270.46,388.48) -- (259.91,400.76) -- (263.93,392.68) -- cycle ;
\draw  [fill={rgb, 255:red, 0; green, 0; blue, 0 }  ,fill opacity=1 ] (420.38,120.76) -- (436.53,122.07) -- (422.3,129.81) -- (428.93,123.68) -- cycle ;
\draw  [fill={rgb, 255:red, 0; green, 0; blue, 0 }  ,fill opacity=1 ] (261.02,289.2) -- (277.16,287.85) -- (264.4,297.81) -- (269.93,290.68) -- cycle ;
\draw  [fill={rgb, 255:red, 0; green, 0; blue, 0 }  ,fill opacity=1 ] (487.32,132.81) -- (502.69,137.91) -- (487.04,142.06) -- (494.93,137.68) -- cycle ;
\draw  [fill={rgb, 255:red, 0; green, 0; blue, 0 }  ,fill opacity=1 ] (470.64,154.2) -- (484.43,162.69) -- (468.24,163.14) -- (476.93,160.68) -- cycle ;
\draw  [fill={rgb, 255:red, 0; green, 0; blue, 0 }  ,fill opacity=1 ] (367.08,157.87) -- (383.26,157.11) -- (370.14,166.6) -- (375.93,159.68) -- cycle ;
\draw  [fill={rgb, 255:red, 0; green, 0; blue, 0 }  ,fill opacity=1 ] (392.43,182.64) -- (408.55,184.17) -- (394.22,191.72) -- (400.93,185.68) -- cycle ;

\draw (480,35) node [anchor=north west][inner sep=0.75pt]    {${\mathbf 1}$};
\draw (660,55) node [anchor=north west][inner sep=0.75pt]    {$a^{4}$};
\draw (800,150) node [anchor=north west][inner sep=0.75pt]    {$a^{2} x$};
\draw (885,280) node [anchor=north west][inner sep=0.75pt]    {$a^{6} x$};
\draw (930,480) node [anchor=north west][inner sep=0.75pt]    {$a$};
\draw (880,650) node [anchor=north west][inner sep=0.75pt]    {$a^{5}$};
\draw (802,780) node [anchor=north west][inner sep=0.75pt]    {$a^{3} x$};
\draw (649,890) node [anchor=north west][inner sep=0.75pt]    {$a^{7} x$};
\draw (485,920) node [anchor=north west][inner sep=0.75pt]    {$a^{2}$};
\draw (300,895.57) node [anchor=north west][inner sep=0.75pt]    {$a^{6}$};
\draw (167,801.57) node [anchor=north west][inner sep=0.75pt]    {$x$};
\draw (40,651.57) node [anchor=north west][inner sep=0.75pt]    {$a^{4} x$};
\draw (20,470) node [anchor=north west][inner sep=0.75pt]    {$a^{3}$};
\draw (50,290) node [anchor=north west][inner sep=0.75pt]    {$a^{7}$};
\draw (160,160) node [anchor=north west][inner sep=0.75pt]    {$ax$};
\draw (300,55) node [anchor=north west][inner sep=0.75pt]    {$a^{5} x$};

\end{tikzpicture}
\caption{$S=\{ a,a^5, a^3x, a^7x\}$}\label{aaa}
\end{subfigure}
\hfill
\begin{subfigure}{0.45\textwidth}

\tikzset{every picture/.style={line width=0.75pt}} 

\begin{tikzpicture}[x=0.2pt,y=0.2pt,yscale=-1,xscale=1]

\draw  [color={rgb, 255:red, 0; green, 0; blue, 0 }  ,draw opacity=1 ][line width=0.75]  (487.6,100.86) .. controls (487.6,93.84) and (493.3,88.15) .. (500.32,88.15) .. controls (507.35,88.15) and (513.04,93.84) .. (513.04,100.86) .. controls (513.04,107.89) and (507.35,113.58) .. (500.32,113.58) .. controls (493.3,113.58) and (487.6,107.89) .. (487.6,100.86) -- cycle ;
\draw  [color={rgb, 255:red, 0; green, 0; blue, 0 }  ,draw opacity=1 ][line width=0.75]  (641.6,131.86) .. controls (641.6,124.84) and (647.3,119.15) .. (654.32,119.15) .. controls (661.35,119.15) and (667.04,124.84) .. (667.04,131.86) .. controls (667.04,138.89) and (661.35,144.58) .. (654.32,144.58) .. controls (647.3,144.58) and (641.6,138.89) .. (641.6,131.86) -- cycle ;
\draw  [color={rgb, 255:red, 0; green, 0; blue, 0 }  ,draw opacity=1 ][line width=0.75]  (770.6,217.86) .. controls (770.6,210.84) and (776.3,205.15) .. (783.32,205.15) .. controls (790.35,205.15) and (796.04,210.84) .. (796.04,217.86) .. controls (796.04,224.89) and (790.35,230.58) .. (783.32,230.58) .. controls (776.3,230.58) and (770.6,224.89) .. (770.6,217.86) -- cycle ;
\draw  [color={rgb, 255:red, 0; green, 0; blue, 0 }  ,draw opacity=1 ][line width=0.75]  (204.6,218.86) .. controls (204.6,211.84) and (210.3,206.15) .. (217.32,206.15) .. controls (224.35,206.15) and (230.04,211.84) .. (230.04,218.86) .. controls (230.04,225.89) and (224.35,231.58) .. (217.32,231.58) .. controls (210.3,231.58) and (204.6,225.89) .. (204.6,218.86) -- cycle ;
\draw  [color={rgb, 255:red, 0; green, 0; blue, 0 }  ,draw opacity=1 ][line width=0.75]  (119.6,346.86) .. controls (119.6,339.84) and (125.3,334.15) .. (132.32,334.15) .. controls (139.35,334.15) and (145.04,339.84) .. (145.04,346.86) .. controls (145.04,353.89) and (139.35,359.58) .. (132.32,359.58) .. controls (125.3,359.58) and (119.6,353.89) .. (119.6,346.86) -- cycle ;
\draw  [color={rgb, 255:red, 0; green, 0; blue, 0 }  ,draw opacity=1 ][line width=0.75]  (87.6,500.86) .. controls (87.6,493.84) and (93.3,488.15) .. (100.32,488.15) .. controls (107.35,488.15) and (113.04,493.84) .. (113.04,500.86) .. controls (113.04,507.89) and (107.35,513.58) .. (100.32,513.58) .. controls (93.3,513.58) and (87.6,507.89) .. (87.6,500.86) -- cycle ;
\draw  [color={rgb, 255:red, 0; green, 0; blue, 0 }  ,draw opacity=1 ][line width=0.75]  (205.6,782.86) .. controls (205.6,775.84) and (211.3,770.15) .. (218.32,770.15) .. controls (225.35,770.15) and (231.04,775.84) .. (231.04,782.86) .. controls (231.04,789.89) and (225.35,795.58) .. (218.32,795.58) .. controls (211.3,795.58) and (205.6,789.89) .. (205.6,782.86) -- cycle ;
\draw  [color={rgb, 255:red, 0; green, 0; blue, 0 }  ,draw opacity=1 ][line width=0.75]  (333.6,869.86) .. controls (333.6,862.84) and (339.3,857.15) .. (346.32,857.15) .. controls (353.35,857.15) and (359.04,862.84) .. (359.04,869.86) .. controls (359.04,876.89) and (353.35,882.58) .. (346.32,882.58) .. controls (339.3,882.58) and (333.6,876.89) .. (333.6,869.86) -- cycle ;
\draw  [color={rgb, 255:red, 0; green, 0; blue, 0 }  ,draw opacity=1 ][line width=0.75]  (488.6,899.86) .. controls (488.6,892.84) and (494.3,887.15) .. (501.32,887.15) .. controls (508.35,887.15) and (514.04,892.84) .. (514.04,899.86) .. controls (514.04,906.89) and (508.35,912.58) .. (501.32,912.58) .. controls (494.3,912.58) and (488.6,906.89) .. (488.6,899.86) -- cycle ;
\draw  [color={rgb, 255:red, 0; green, 0; blue, 0 }  ,draw opacity=1 ][line width=0.75]  (770.6,782.86) .. controls (770.6,775.84) and (776.3,770.15) .. (783.32,770.15) .. controls (790.35,770.15) and (796.04,775.84) .. (796.04,782.86) .. controls (796.04,789.89) and (790.35,795.58) .. (783.32,795.58) .. controls (776.3,795.58) and (770.6,789.89) .. (770.6,782.86) -- cycle ;
\draw  [color={rgb, 255:red, 0; green, 0; blue, 0 }  ,draw opacity=1 ][line width=0.75]  (857.6,653.86) .. controls (857.6,646.84) and (863.3,641.15) .. (870.32,641.15) .. controls (877.35,641.15) and (883.04,646.84) .. (883.04,653.86) .. controls (883.04,660.89) and (877.35,666.58) .. (870.32,666.58) .. controls (863.3,666.58) and (857.6,660.89) .. (857.6,653.86) -- cycle ;
\draw  [color={rgb, 255:red, 0; green, 0; blue, 0 }  ,draw opacity=1 ][line width=0.75]  (886.6,499.86) .. controls (886.6,492.84) and (892.3,487.15) .. (899.32,487.15) .. controls (906.35,487.15) and (912.04,492.84) .. (912.04,499.86) .. controls (912.04,506.89) and (906.35,512.58) .. (899.32,512.58) .. controls (892.3,512.58) and (886.6,506.89) .. (886.6,499.86) -- cycle ;
\draw  [color={rgb, 255:red, 0; green, 0; blue, 0 }  ,draw opacity=1 ][line width=0.75]  (335.6,130.86) .. controls (335.6,123.84) and (341.3,118.15) .. (348.32,118.15) .. controls (355.35,118.15) and (361.04,123.84) .. (361.04,130.86) .. controls (361.04,137.89) and (355.35,143.58) .. (348.32,143.58) .. controls (341.3,143.58) and (335.6,137.89) .. (335.6,130.86) -- cycle ;
\draw  [color={rgb, 255:red, 0; green, 0; blue, 0 }  ,draw opacity=1 ][line width=0.75]  (641.6,869.86) .. controls (641.6,862.84) and (647.3,857.15) .. (654.32,857.15) .. controls (661.35,857.15) and (667.04,862.84) .. (667.04,869.86) .. controls (667.04,876.89) and (661.35,882.58) .. (654.32,882.58) .. controls (647.3,882.58) and (641.6,876.89) .. (641.6,869.86) -- cycle ;
\draw  [color={rgb, 255:red, 0; green, 0; blue, 0 }  ,draw opacity=1 ][line width=0.75]  (119.6,653.86) .. controls (119.6,646.84) and (125.3,641.15) .. (132.32,641.15) .. controls (139.35,641.15) and (145.04,646.84) .. (145.04,653.86) .. controls (145.04,660.89) and (139.35,666.58) .. (132.32,666.58) .. controls (125.3,666.58) and (119.6,660.89) .. (119.6,653.86) -- cycle ;
\draw  [color={rgb, 255:red, 0; green, 0; blue, 0 }  ,draw opacity=1 ][line width=0.75]  (857.6,347.86) .. controls (857.6,340.84) and (863.3,335.15) .. (870.32,335.15) .. controls (877.35,335.15) and (883.04,340.84) .. (883.04,347.86) .. controls (883.04,354.89) and (877.35,360.58) .. (870.32,360.58) .. controls (863.3,360.58) and (857.6,354.89) .. (857.6,347.86) -- cycle ;
\draw    (509,111.17) -- (890.58,491.42) ;
\draw    (509,111.17) -- (862.58,645.42) ;
\draw    (509,111.17) -- (783.32,770.15) ;
\draw    (509,111.17) -- (654.32,857.15) ;
\draw    (660.58,143.42) -- (890.58,491.42) ;
\draw    (660.58,143.42) -- (862.58,645.42) ;
\draw    (660.58,143.42) -- (783.32,770.15) ;
\draw    (660.58,143.42) -- (654.32,857.15) ;
\draw    (783.32,230.58) -- (890.58,491.42) ;
\draw    (783.32,230.58) -- (862.58,645.42) ;
\draw    (783.32,230.58) -- (783.32,770.15) ;
\draw    (783.32,230.58) -- (654.32,857.15) ;
\draw    (864.58,359.42) -- (890.58,491.42) ;
\draw    (864.58,359.42) -- (862.58,645.42) ;
\draw    (864.58,359.42) -- (783.32,770.15) ;
\draw    (864.58,359.42) -- (654.32,857.15) ;
\draw    (496.71,888.52) -- (108.58,510.42) ;
\draw    (496.71,888.52) -- (138.47,357.38) ;
\draw    (496.71,888.52) -- (216.64,231.96) ;
\draw    (496.71,888.52) -- (344.88,143.84) ;
\draw    (344.85,857.6) -- (108.58,510.42) ;
\draw    (344.85,857.6) -- (138.47,357.38) ;
\draw    (344.85,857.6) -- (216.64,231.96) ;
\draw    (344.85,857.6) -- (344.88,143.84) ;
\draw    (221.35,771.5) -- (108.58,510.42) ;
\draw    (221.35,771.5) -- (138.47,357.38) ;
\draw    (221.35,771.5) -- (216.64,231.96) ;
\draw    (221.35,771.5) -- (344.88,143.84) ;
\draw    (138.97,643.38) -- (108.58,510.42) ;
\draw    (138.97,643.38) -- (138.47,357.38) ;
\draw    (138.97,643.38) -- (216.64,231.96) ;
\draw    (138.97,643.38) -- (344.88,143.84) ;
\draw    (890.07,508.46) -- (510.56,890.78) ;
\draw    (890.07,508.46) -- (356.5,863.07) ;
\draw    (890.07,508.46) -- (231.62,784.05) ;
\draw    (890.07,508.46) -- (144.38,655.22) ;
\draw    (858.12,660.11) -- (510.56,890.78) ;
\draw    (858.12,660.11) -- (356.5,863.07) ;
\draw    (858.12,660.11) -- (231.62,784.05) ;
\draw    (858.12,660.11) -- (144.38,655.22) ;
\draw    (771.18,783.01) -- (510.56,890.78) ;
\draw    (771.18,783.01) -- (356.5,863.07) ;
\draw    (771.18,783.01) -- (231.62,784.05) ;
\draw    (771.18,783.01) -- (144.38,655.22) ;
\draw    (642.51,864.52) -- (510.56,890.78) ;
\draw    (642.51,864.52) -- (356.5,863.07) ;
\draw    (642.51,864.52) -- (231.62,784.05) ;
\draw    (642.51,864.52) -- (144.38,655.22) ;
\draw    (111.88,493.69) -- (491.27,111.24) ;
\draw    (111.88,493.69) -- (645.33,138.89) ;
\draw    (111.88,493.69) -- (770.24,217.87) ;
\draw    (111.88,493.69) -- (857.53,346.67) ;
\draw    (143.79,342.03) -- (491.27,111.24) ;
\draw    (143.79,342.03) -- (645.33,138.89) ;
\draw    (143.79,342.03) -- (770.24,217.87) ;
\draw    (143.79,342.03) -- (857.53,346.67) ;
\draw    (230.68,219.1) -- (491.27,111.24) ;
\draw    (230.68,219.1) -- (645.33,138.89) ;
\draw    (230.68,219.1) -- (770.24,217.87) ;
\draw    (230.68,219.1) -- (857.53,346.67) ;
\draw    (359.33,137.54) -- (491.27,111.24) ;
\draw    (359.33,137.54) -- (645.33,138.89) ;
\draw    (359.33,137.54) -- (770.24,217.87) ;
\draw    (359.33,137.54) -- (857.53,346.67) ;
\draw  [fill={rgb, 255:red, 0; green, 0; blue, 0 }  ,fill opacity=1 ] (539.01,239.18) -- (537.41,255.29) -- (529.92,240.94) -- (535.93,247.68) -- cycle ;
\draw  [fill={rgb, 255:red, 0; green, 0; blue, 0 }  ,fill opacity=1 ] (563.3,228.75) -- (564.85,244.87) -- (554.73,232.23) -- (561.93,237.68) -- cycle ;
\draw  [fill={rgb, 255:red, 0; green, 0; blue, 0 }  ,fill opacity=1 ] (663.55,281.91) -- (658.95,297.44) -- (654.29,281.92) -- (658.93,289.68) -- cycle ;
\draw  [fill={rgb, 255:red, 0; green, 0; blue, 0 }  ,fill opacity=1 ] (701.1,324.21) -- (699.32,340.31) -- (692,325.87) -- (697.93,332.68) -- cycle ;
\draw  [fill={rgb, 255:red, 0; green, 0; blue, 0 }  ,fill opacity=1 ] (765.5,537.56) -- (749.35,536.31) -- (763.55,528.52) -- (756.93,534.68) -- cycle ;
\draw  [fill={rgb, 255:red, 0; green, 0; blue, 0 }  ,fill opacity=1 ] (818.44,388.35) -- (816.01,404.36) -- (809.28,389.63) -- (814.93,396.68) -- cycle ;
\draw  [fill={rgb, 255:red, 0; green, 0; blue, 0 }  ,fill opacity=1 ] (815.42,295.76) -- (816.76,311.9) -- (806.8,299.13) -- (813.93,304.68) -- cycle ;
\draw  [fill={rgb, 255:red, 0; green, 0; blue, 0 }  ,fill opacity=1 ] (868.75,486.03) -- (863.75,501.43) -- (859.5,485.81) -- (863.93,493.68) -- cycle ;
\draw  [fill={rgb, 255:red, 0; green, 0; blue, 0 }  ,fill opacity=1 ] (718.51,663.6) -- (703.18,658.38) -- (718.87,654.35) -- (710.93,658.68) -- cycle ;
\draw  [fill={rgb, 255:red, 0; green, 0; blue, 0 }  ,fill opacity=1 ] (766.93,564.5) -- (750.94,567.03) -- (762.93,556.15) -- (757.93,563.68) -- cycle ;
\draw  [fill={rgb, 255:red, 0; green, 0; blue, 0 }  ,fill opacity=1 ] (880.85,416.12) -- (879.54,432.27) -- (871.8,418.04) -- (877.93,424.68) -- cycle ;
\draw  [fill={rgb, 255:red, 0; green, 0; blue, 0 }  ,fill opacity=1 ] (678.42,699.78) -- (662.31,698.12) -- (676.7,690.69) -- (669.93,696.68) -- cycle ;
\draw  [fill={rgb, 255:red, 0; green, 0; blue, 0 }  ,fill opacity=1 ] (708.93,813.5) -- (692.94,816.03) -- (704.93,805.15) -- (699.93,812.68) -- cycle ;
\draw  [fill={rgb, 255:red, 0; green, 0; blue, 0 }  ,fill opacity=1 ] (708.39,799.87) -- (692.3,798.04) -- (706.76,790.76) -- (699.93,796.68) -- cycle ;
\draw  [fill={rgb, 255:red, 0; green, 0; blue, 0 }  ,fill opacity=1 ] (159.48,638.7) -- (156.3,622.82) -- (167.66,634.37) -- (159.93,629.68) -- cycle ;
\draw  [fill={rgb, 255:red, 0; green, 0; blue, 0 }  ,fill opacity=1 ] (120.36,582.34) -- (121.02,566.15) -- (129.33,580.06) -- (122.93,573.68) -- cycle ;
\draw  [fill={rgb, 255:red, 0; green, 0; blue, 0 }  ,fill opacity=1 ] (184.68,611.1) -- (186.63,595.03) -- (193.8,609.55) -- (187.93,602.68) -- cycle ;
\draw  [fill={rgb, 255:red, 0; green, 0; blue, 0 }  ,fill opacity=1 ] (134.59,515.6) -- (138.66,499.92) -- (143.84,515.26) -- (138.93,507.68) -- cycle ;
\draw  [fill={rgb, 255:red, 0; green, 0; blue, 0 }  ,fill opacity=1 ] (523.68,868.33) -- (508.17,863.65) -- (523.71,859.08) -- (515.93,863.68) -- cycle ;
\draw  [fill={rgb, 255:red, 0; green, 0; blue, 0 }  ,fill opacity=1 ] (579.58,881.31) -- (563.4,880.54) -- (577.36,872.32) -- (570.93,878.68) -- cycle ;
\draw  [fill={rgb, 255:red, 0; green, 0; blue, 0 }  ,fill opacity=1 ] (302.88,678.18) -- (304.45,662.06) -- (311.96,676.4) -- (305.93,669.68) -- cycle ;
\draw  [fill={rgb, 255:red, 0; green, 0; blue, 0 }  ,fill opacity=1 ] (340.35,717.46) -- (344.89,701.92) -- (349.6,717.41) -- (344.93,709.68) -- cycle ;
\draw  [fill={rgb, 255:red, 0; green, 0; blue, 0 }  ,fill opacity=1 ] (467.84,768.17) -- (469.48,752.05) -- (476.94,766.43) -- (470.93,759.68) -- cycle ;
\draw  [fill={rgb, 255:red, 0; green, 0; blue, 0 }  ,fill opacity=1 ] (439.31,766.56) -- (438.22,750.41) -- (447.98,763.33) -- (440.93,757.68) -- cycle ;
\draw  [fill={rgb, 255:red, 0; green, 0; blue, 0 }  ,fill opacity=1 ] (322.44,301.61) -- (338.55,303.2) -- (324.2,310.69) -- (330.93,304.68) -- cycle ;
\draw  [fill={rgb, 255:red, 0; green, 0; blue, 0 }  ,fill opacity=1 ] (237.5,464.45) -- (253.58,466.35) -- (239.08,473.56) -- (245.93,467.68) -- cycle ;
\draw  [fill={rgb, 255:red, 0; green, 0; blue, 0 }  ,fill opacity=1 ] (276.59,338.41) -- (291.67,344.33) -- (275.81,347.63) -- (283.93,343.68) -- cycle ;
\draw  [fill={rgb, 255:red, 0; green, 0; blue, 0 }  ,fill opacity=1 ] (234,437.35) -- (250.11,435.72) -- (237.52,445.91) -- (242.93,438.68) -- cycle ;
\draw  [fill={rgb, 255:red, 0; green, 0; blue, 0 }  ,fill opacity=1 ] (420.38,120.76) -- (436.53,122.07) -- (422.3,129.81) -- (428.93,123.68) -- cycle ;
\draw  [fill={rgb, 255:red, 0; green, 0; blue, 0 }  ,fill opacity=1 ] (487.32,132.81) -- (502.69,137.91) -- (487.04,142.06) -- (494.93,137.68) -- cycle ;
\draw  [fill={rgb, 255:red, 0; green, 0; blue, 0 }  ,fill opacity=1 ] (367.08,157.87) -- (383.26,157.11) -- (370.14,166.6) -- (375.93,159.68) -- cycle ;
\draw  [fill={rgb, 255:red, 0; green, 0; blue, 0 }  ,fill opacity=1 ] (392.43,182.64) -- (408.55,184.17) -- (394.22,191.72) -- (400.93,185.68) -- cycle ;

\draw (480,35) node [anchor=north west][inner sep=0.75pt]    {${\mathbf 1}$};
\draw (660,55) node [anchor=north west][inner sep=0.75pt]    {$a^{4}$};
\draw (800,150) node [anchor=north west][inner sep=0.75pt]    {$a^{2} x$};
\draw (885,280) node [anchor=north west][inner sep=0.75pt]    {$a^{6} x$};
\draw (930,480) node [anchor=north west][inner sep=0.75pt]    {$a$};
\draw (880,650) node [anchor=north west][inner sep=0.75pt]    {$a^{5}$};
\draw (802,780) node [anchor=north west][inner sep=0.75pt]    {$a^{3} x$};
\draw (649,890) node [anchor=north west][inner sep=0.75pt]    {$a^{7} x$};
\draw (485,920) node [anchor=north west][inner sep=0.75pt]    {$a^{2}$};
\draw (300,895.57) node [anchor=north west][inner sep=0.75pt]    {$a^{6}$};
\draw (167,801.57) node [anchor=north west][inner sep=0.75pt]    {$x$};
\draw (40,651.57) node [anchor=north west][inner sep=0.75pt]    {$a^{4} x$};
\draw (20,470) node [anchor=north west][inner sep=0.75pt]    {$a^{3}$};
\draw (50,290) node [anchor=north west][inner sep=0.75pt]    {$a^{7}$};
\draw (160,160) node [anchor=north west][inner sep=0.75pt]    {$ax$};
\draw (300,55) node [anchor=north west][inner sep=0.75pt]    {$a^{5} x$};

\end{tikzpicture}

\caption{$S=\{ a,a^3,a^5,a^7, a^3x, a^7x \}$} \label{bbb}
\end{subfigure}
\caption{The mixed graph ${\rm Cay}(M_{16}, S)$}
\end{figure}
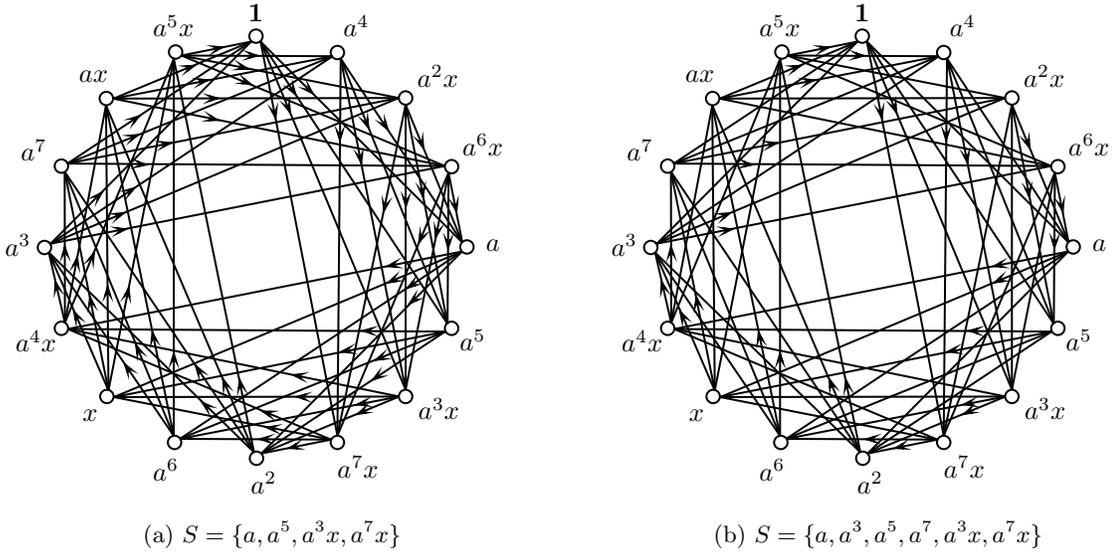


\section{Gaussian integral normal mixed Cayley graphs}\label{gussian-integral}

In \cite{kadyan2021integralAbelian}, the authors proved that if $\Gamma$ is an abelian group, then $\llbracket x \rrbracket \cup \llbracket x^{-1} \rrbracket =[x]$ for each $x \in \Gamma(4)$. Note that this result and its proof also hold for non-abelian group. In the subsequent discussion, we use this fact for non-abelian group.

Let $S$ be a union of some conjugacy classes of a finite group $\Gamma$ that does not contain ${\mathbf 1}$, and let ${\rm Irr}(\Gamma)=\{ \chi_1,\ldots,\chi_h\}$. Consider the function $f \colon \Gamma \rightarrow \{0,1\}$ defined by 
\[f(s)= \left\{ \begin{array}{rl}
		1 & \mbox{if } s\in S \\
		0 &   \mbox{otherwise} 
	\end{array}\right.\]
in Theorem~\ref{EigNorColCayMix1}. We see that $\frac{1}{\chi_j({\mathbf 1})}\sum\limits_{s \in S} \chi_{j}(s)$ is an eigenvalue of the normal mixed Cayley graph ${\rm Cay}(\Gamma, S)$ for each $j \in \{ 1,\ldots,h\}$. Indeed, all the eigenvalues of ${\rm Cay}(\Gamma, S)$ are of this form. For each $j \in \{ 1,\ldots,h\}$, define
$$f_{j}(S) :=\frac{1}{\chi_j({\mathbf 1})} \sum_{s \in S \setminus \overline{S}} \chi_{j}(s) + \frac{1}{2\chi_j({\mathbf 1})} \sum_{s \in \overline{S}\cup \overline{S}^{-1}} \chi_{j}(s)$$ \textnormal{and} $$ g_{j}(S):= \frac{{\mathbf i}}{2\chi_j({\mathbf 1})} \sum_{s \in \overline{S}}  (\chi_{j}(s) - \chi_{j}(s^{-1})).$$ 
We have 
\begin{equation}\label{eqEisenIntPart}
\begin{split}
\frac{1}{\chi_j({\mathbf 1})} \sum_{s \in S} \chi_{j}(s)= f_{j}(S)- {\mathbf i} g_{j}(S) \text{ for each } j \in \{ 1,\ldots,h\}.
\end{split}
\end{equation}

Therefore, the normal mixed Cayley graph ${\rm Cay}(\Gamma,S)$ is Gaussian integral if and only if $f_{j}(S)$ and $g_{j}(S)$ are integers for each $j \in \{ 1,\ldots,h\}$.

\begin{theorem}\label{MinCharacEisensteinIntegCh4}
Let $\Gamma$ be a finite group. If the normal mixed Cayley graph  ${\rm Cay}(\Gamma,S)$ is Gaussian integral, then it is H-integral.
\end{theorem}
\begin{proof}
Let the normal mixed Cayley graph ${\rm Cay}(\Gamma,S)$ be Gaussian integral. Therefore, $f_{j}(S)$ and $g_{j}(S)$ are integers for each $j \in \{ 1,\ldots,h\}$. Note that $2g_j(S)$ is an integer H-eigenvalue of the normal oriented Cayley graph ${\rm Cay}(\Gamma,\overline{S} )$ for each $j \in \{ 1,\ldots,h\}$. By Theorem \ref{NorOriCayGraphInteg}, we get $\overline{S}\in \mathbb{D}(\Gamma)$. Let $\overline{S}=\text{\footnotesize$\bigcup\limits_{j=1}^k$}\llbracket x_j \rrbracket=\text{\footnotesize$\bigcup\limits_{j=1}^s$} {\rm Cl}(y_j)$ for some $x_1,\ldots,x_k,y_1,\ldots,y_s \in \Gamma(4)$. Thus  $$\overline{S} \cup \overline{S}^{-1} = \text{\footnotesize$\bigcup\limits_{j=1}^k$} \left(\llbracket x_j \rrbracket \cup \llbracket x_j^{-1} \rrbracket\right) =  \text{\footnotesize$\bigcup\limits_{j=1}^k$} [x_j]=\text{\footnotesize$\bigcup\limits_{j=1}^s$} \left( {\rm Cl}(y_j)\cup {\rm Cl}(y_j^{-1})\right) \in \mathbb{B}(\Gamma).$$ Note that  $\frac{1}{\chi_j({\mathbf 1})} \sum\limits_{s \in  \overline{S}\cup \overline{S}^{-1} } \chi_j(s)$ is an eigenvalue of the normal simple Cayley graph \linebreak[4] ${\rm Cay}(\Gamma,\overline{S} \cup \overline{S}^{-1} )$ for each $j \in \{ 1,\ldots,h\}$. By Theorem~\ref{NorMixCayGraphInteg}, $ \frac{1}{\chi_j({\mathbf 1})}  \sum\limits_{s \in  \overline{S}\cup \overline{S}^{-1} } \chi_j(s)$ is an integer for each $j \in \{ 1,\ldots,h\}$. Therefore
$$ \frac{1}{\chi_j({\mathbf 1})}  \sum\limits_{s \in S\setminus \overline{S}} \chi_j(s) = f_j(S) -  \frac{1}{2 \chi_j({\mathbf 1}) } \sum\limits_{s \in  \overline{S}\cup \overline{S}^{-1} } \chi_j(s),$$ 
and hence $\frac{1}{\chi_j({\mathbf 1})}  \sum\limits_{s \in S\setminus \overline{S}} \chi_j(s)$  is a rational number. Thus, the eigenvalue $ \frac{1}{\chi_j({\mathbf 1})} \sum\limits_{s \in S\setminus \overline{S}} \chi_j(s)$ of the normal simple Cayley graph ${\rm Cay}(\Gamma,S\setminus \overline{S} )$ is a rational algebraic integer, and hence it is an integer for each $j \in \{1,\ldots,h\}$. Hence by Theorem~\ref{NorMixCayGraphInteg}, we get $S\setminus \overline{S} \in \mathbb{B}(\Gamma)$. Now the result follows from Theorem~\ref{NormCayIntegChara}.
\end{proof}

\begin{lema}\label{NewLemmaEquivaTrans1} Let $x \in \Gamma$ and ${\rm ord}(x)=2^{t}m$. If $t \geq 2$ and $m$ is odd, then the following assertions hold.
\begin{enumerate}[label=(\roman*)]
\item $\llbracket x \rrbracket  = \left\{ \begin{array}{ll}
			 x^{3m} [x^2] & \mbox{if } m \equiv 1 \Mod 4  \\
			x^{m} [x^2] & \mbox{if } m \equiv 3 \Mod 4.
		\end{array}\right. $
\item $\llbracket x^{-1} \rrbracket  = \left\{ \begin{array}{ll}
			 x^{m} [x^2] & \mbox{if } m \equiv 1 \Mod 4  \\
			x^{3m} [x^2] & \mbox{if } m \equiv 3 \Mod 4.
		\end{array}\right. $
\item  $[x]=x^m[x^2] \cup x^{3m} [x^2]$.
\end{enumerate}
\end{lema}
\begin{proof}
\begin{enumerate}[label=(\roman*)]
\item Assume that $t \geq 2$ and $k=2^{t}m$. Let $m \equiv 1 \Mod 4$ and $x^{3m+2r} \in x^{3m}[x^2]$ for some $r \in G_{\frac{k}{2}}(1)$. Now $\gcd(r, \frac{k}{2})=1$, and it implies that $\gcd(3m+2r, k)=1$ and  $3m+2r \equiv 1 \Mod 4$. We have $x^{3m+2r} \in \llbracket x \rrbracket$. Thus $x^{3m} [x^2] \subseteq \llbracket x \rrbracket$. Now since the sizes of $\llbracket x \rrbracket$ and $x^{3m} [x^2]$ are equal, we have $\llbracket x \rrbracket = x^{3m} [x^2]$. Similarly, if $m \equiv 3 \Mod 4$ then $\llbracket x \rrbracket = x^{m} [x^2]$.
\item The proof of this part is similar to the proof of Part (i). For the sake of completeness, we provide the proof. Assume that $t \geq 2$ and $k=2^{t}m$. Let $m \equiv 1 \Mod 4$  and $x^{m+2r} \in x^{m}[x^2]$ for some $r \in G_{\frac{k}{2}}(1)$. Now $\gcd(r, \frac{k}{2})=1$ implies that $\gcd(m+2r, k)=1$ and  $m+2r \equiv 3 \Mod 4$. We have $x^{m+2r} \in \llbracket x^{-1} \rrbracket$. Thus $x^{m} [x^2] \subseteq \llbracket x^{-1} \rrbracket$. Now since the sizes of $\llbracket x^{-1} \rrbracket$ and $x^{m} [x^2]$ are equal, we have $\llbracket x^{-1} \rrbracket = x^{m} [x^2]$. Similarly, if $m \equiv 3 \Mod 4$ then $\llbracket x^{-1} \rrbracket = x^{3m} [x^2]$.
\item Using Part (i), Part (ii) and $[x]=\llbracket x \rrbracket \cup \llbracket x^{-1} \rrbracket$, we get the result in desired form. \qedhere
\end{enumerate}
\end{proof}

For $ x\in \Gamma$, define $S_x^1:=\text{\footnotesize$\bigcup\limits_{s\in {\rm Cl}(x)}$} [ s ]$. We see that if $m=\ord(x)$, then 
\[S_x^1=\{g^{-1}x^r g\colon g\in \Gamma, r\in G_m(1)\}= \text{\footnotesize$\bigcup\limits_{s\in [x]}$}{\rm Cl}(s).\]
The set $S_x^1$ is also known as the rational conjugacy class\index{rational conjugacy class} of $x$. See~\cite{foster2016spectra} for details. For each $y\in S_x^1$, it is clear that ${\rm Cl}(y), [y]\subseteq S_x^1$. Now let $A$ be a symmetric subset of $\Gamma$ such that $x\in A$, and ${\rm Cl}(a), [a]\subseteq A$ for each $a\in A$. Let $g^{-1}x^r g\in S_x^1$, where $g\in \Gamma$, $r\in G_m(1)$ and $m=\ord(x)$. As $[x]\subseteq A$, we have $x^r\in A$. Now ${\rm Cl}(x^r)\subseteq A$, and so $g^{-1}x^r g\in A$. Thus $S_x^1 \subseteq A$, and therefore $S_x^1$ is the smallest symmetric subset of $\Gamma$ containing $x$ that is closed under both conjugacy and the equivalence relation $\sim$. Considering each of the  repeated equivalence classes, if any, only once  in $\text{\footnotesize$\bigcup\limits_{s\in {\rm Cl}(x)}$} [ s ]$, we can write $S_x^1=\text{\footnotesize$\bigcup\limits_{i=1}^{\ell}$}[x_i]$, where the equivalence classes $[x_1], \ldots, [x_{\ell}]$ are distinct. We state this fact in the next lemma.

\begin{lema}\label{SymbolSetChara1}
If $x \in \Gamma$, then there exist distinct equivalence classes $[x_1], \ldots, [x_{\ell}]$  such that $S_x^1= \text{\footnotesize$\bigcup\limits_{i=1}^{\ell}$}[x_i]$, where $x_1,\ldots,x_{\ell}\in {\rm Cl}(x)$.
\end{lema}

\begin{lema}\label{partition1}If $y\in S_x^1$, then $S_y^1=S_x^1$.
\end{lema}
\begin{proof}Let $y\in S_x^1$, so that $y=g^{-1}x^r g$ for some $g\in \Gamma$ and $r\in G_m(1)$, where $m=\ord(x)$. We see that $\ord(y)=\ord(x)=m$. Now let $z\in S_y^1$. Then $z=h^{-1}y^t h$ for some $h\in \Gamma$ and $t\in G_m(1)$. This gives $z=h^{-1}y^t h = h^{-1}g^{-1}x^{rt} g h\in S_x^1$. Conversely, let $w\in S_x^1$ so that $w=h^{-1}x^t h$ for some $h\in \Gamma$ and $t\in G_m(1)$. Therefore \[w=h^{-1}x^t h = (h^{-1}g)g^{-1}(x^r)^{r^{-1}t}g( g^{-1} h) =(h^{-1}g)y^{r^{-1}t}  ( g^{-1} h)\in S_y^1.\] Here $r^{-1}$ is the multiplicative inverse of $r$ in the group $G_m(1)$. Hence we conclude that $S_y^1=S_x^1$.
\end{proof}
Due to Lemma~\ref{partition1}, the sets $S_x^1$ and $S_y^1$ are either disjoint or equal. Hence the class of distinct subsets of $\Gamma$ of the form $S_x^1$ is a partition of $\Gamma$.

\begin{lema}\label{NewLemmaEquivaTrans2} Let $x\in \Gamma(4)$. If $S_x^1 =[ x_1 ] \cup \cdots \cup [ x_{\ell} ]$ for some $x_1,\ldots,x_{\ell}\in {\rm Cl}(x)$, then $S_{x^2}^1= [ x_1^2 ] \cup \cdots \cup [ x_{\ell}^2 ]$.
\end{lema}
\begin{proof}
Let $m={\rm ord}(x)$ and $S_x^1 = [ x_1 ] \cup \cdots \cup [ x_{\ell} ]$ for some $x_1\ldots,x_{\ell}\in {\rm Cl}(x)$. Assume that the sets $[ x_1 ], \ldots , [ x_{\ell} ]$ are all distinct. We see that
\begin{equation*}
\begin{split}
S_{x^2}^1&=\left\{g^{-1}x^{2r} g\colon g\in \Gamma, r\in G_{\frac{m}{2}}(1)\right\}\\
&=\left\{g^{-1}x^{2r} g\colon g\in \Gamma, r\in G_{\frac{m}{2}}(1)\right\} \cup \left\{g^{-1}x^{2(\frac{m}{2}+r)} g\colon g\in \Gamma, r\in G_{\frac{m}{2}}(1)\right\}\\
&=\left\{g^{-1}x^{2r} g\colon g\in \Gamma, r\in G_{m}(1), r <\frac{m}{2} \right\} \cup \left\{g^{-1}x^{2t} g\colon g\in \Gamma, t\in G_{m}(1), t >\frac{m}{2}\right\}\\
&=\{g^{-1}x^{2r} g\colon g\in \Gamma, r\in G_{m}(1) \}\\
&=\{y^2\colon y\in S_x^1\}.
\end{split}
\end{equation*}
Now noting that $\{s^2\colon s\in [x]\} =[x^2]$ and $S_x^1 =[ x_1 ] \cup \cdots \cup [ x_{\ell} ]$, we have $S_{x^2}^1= [ x_1^2 ] \cup \cdots \cup [ x_{\ell}^2 ]$.
\end{proof}

Let $x\in \Gamma(4)$ be an element of order $m$. The element $x$ is said to be \textit{admissible}\index{admissible} if $x^r \not\in {\rm Cl}(x)$ for all $r \in G_m^3(1)$. 
The following lemma characterizes admissible elements in terms of skew-symmetric sets. 
\begin{lema}\label{iff1}
If $x\in \Gamma(4)$, then $x$ is admissible if and only if the set $\text{\footnotesize$\bigcup\limits_{s\in {\rm Cl}(x)}$} \llbracket s \rrbracket$ is  skew-symmetric.
\end{lema} 
\begin{proof}We see that if $m=\ord(x)$, then 
\[\text{\footnotesize$\bigcup\limits_{s\in {\rm Cl}(x)}$} \llbracket s \rrbracket=\{g^{-1}x^r g\colon g\in \Gamma, r\in G_m^1(1)\}= \text{\footnotesize$\bigcup\limits_{s\in \llbracket x \rrbracket}$}{\rm Cl}(s).\]
Assume that $x$ is not admissible, so that $x^r\in {\rm Cl}(x)$ for some $r\in G_m^3(1)$. As $m-r\in G_m^1(1)$ and ${\rm Cl}(x)\subseteq \text{\footnotesize$\bigcup\limits_{s\in {\rm Cl}(x)}$} \llbracket s \rrbracket$, we find that $x^r,x^{m-r}\in \text{\footnotesize$\bigcup\limits_{s\in {\rm Cl}(x)}$} \llbracket s \rrbracket$. Hence $\text{\footnotesize$\bigcup\limits_{s\in {\rm Cl}(x)}$} \llbracket s \rrbracket$ is not skew-symmetric. 

Now assume that $\text{\footnotesize$\bigcup\limits_{s\in {\rm Cl}(x)}$} \llbracket s \rrbracket$ is not skew-symmetric. Then there is an $y=g^{-1}x^r g\in \text{\footnotesize$\bigcup\limits_{s\in {\rm Cl}(x)}$} \llbracket s \rrbracket$ for some $r\in G_m^1(1)$ such that $y^{-1}\in \text{\footnotesize$\bigcup\limits_{s\in {\rm Cl}(x)}$} \llbracket s \rrbracket$. Therefore $g^{-1}x^{m-r} g=y^{-1}=h^{-1}x^k h$ for some $h\in \Gamma, k\in G_m^1(1)$. Let $t\in G_m(1)$ be the multlipicative inverse of $m-r$. We have $g^{-1}x^{(m-r)t} g=h^{-1}x^{kt} h$, and it gives $x^{kt}=hg^{-1}x gh^{-1}\in {\rm Cl}(x)$. Since $(m-r)t\equiv 1 \Mod 4$ and $m-r\in G_m^3(1)$, we have that $t\in G_m^3(1)$. Thus $kt\in G_m^3(1)$ with $x^{kt}\in {\rm Cl}(x)$, giving that $x$ is not admissible. 
\end{proof}


Let $ x\in \Gamma(4)$ be admissible, and define $S_x^4:=\text{\footnotesize$\bigcup\limits_{s\in {\rm Cl}(x)}$} \llbracket s \rrbracket$. 
The structure and properties of the set $S_x^4$ are similar to those of $S_x^1$. If $\Gamma$ is abelian, then $S_x^4=\llbracket x \rrbracket$ for each $x \in \Gamma(4)$. 

For each $y\in S_x^4$, it is clear that ${\rm Cl}(y), \llbracket y \rrbracket\subseteq S_x^4$. Now let $A$ be a skew-symmetric subset of $\Gamma$ containing an admissible element $x$, and ${\rm Cl}(a), \llbracket a \rrbracket\subseteq A$ for each $a\in A$. It is easy to see that $S_x^4 \subseteq A$. Thus, $S_x^4$ is the smallest skew-symmetric subset of $\Gamma$ containing $x$ that is closed under both conjugacy and the equivalence relation $\approx$. Considering each of the  repeated equivalence classes, if any, only once in $\text{\footnotesize$\bigcup\limits_{s\in {\rm Cl}(x)}$}\llbracket s \rrbracket$, we can write $S_x^4=\text{\footnotesize$\bigcup\limits_{i=1}^r$}\llbracket y_i \rrbracket$, where the equivalence classes $\llbracket y_1 \rrbracket, \ldots, \llbracket y_r \rrbracket$ are distinct. We state this fact in the next lemma.

\begin{lema}\label{SymbolSetChara2}
If $x$ is an admissible element in $\Gamma(4)$, then there are distinct equivalence classes $\llbracket y_1 \rrbracket, \ldots, \llbracket y_r \rrbracket$ such that $S_x^4= \text{\footnotesize$\bigcup\limits_{i=1}^r$}\llbracket y_i \rrbracket$, where $y_1,\ldots,y_r\in {\rm Cl}(x)$.
\end{lema}

\begin{lema}\label{partition2}If $y\in S_x^4$, then $S_y^4=S_x^4$.
\end{lema}
\begin{proof}Let $y\in S_x^4$, so that $y=g^{-1}x^r g$ for some $g\in \Gamma$ and $r\in G_m^1(1)$, where $m=\ord(x)$. We see that $\ord(y)=\ord(x)=m$. Now let $z\in S_y^4$. Then $z=h^{-1}y^t h$ for some $h\in \Gamma$ and $t\in G_m^1(1)$. This gives $z=h^{-1}y^t h = h^{-1}g^{-1}x^{rt} g h\in S_x^4$. Conversely, let $w\in S_x^4$ so that $w=h^{-1}x^t h$ for some $h\in \Gamma$ and $t\in G_m^1(1)$. Therefore 
\[w=h^{-1}x^t h = (h^{-1}g)g^{-1}(x^r)^{r^{-1}t}g( g^{-1} h) =(h^{-1}g)y^{r^{-1}t}  ( g^{-1} h)\in S_y^4.\]
Here $r^{-1}$ is the multiplicative inverse of $r$ in the subgroup $G_m^1(1)$. Thus we conclude that $S_y^4=S_x^4$.
\end{proof}
Due to Lemma~\ref{partition2}, the sets $S_x^4$ and $S_y^4$ are either disjoint or equal. 

\begin{lema} If $x\in \Gamma(4)$ is admissible, then $S_x^4 \cup S_{x^{-1}}^4=S_x^1$.
\end{lema}
\begin{proof}Let $m=\ord(x)$. We have
\begin{align*}
S_x^4 \cup S_{x^{-1}}^4= & \{g^{-1}x^{r} g\colon g\in \Gamma, r\in G_m^1(1)\} \cup \{g^{-1}x^{-r} g\colon g\in \Gamma, r\in G_m^1(1)\}\\
=& \{g^{-1}x^{r} g\colon g\in \Gamma, r\in G_m^1(1)\} \cup \{g^{-1}x^{r} g\colon g\in \Gamma, r\in G_m^3(1)\}\\
=& \{g^{-1}x^{r} g\colon g\in \Gamma, r\in G_m(1)\}\\
=& S_x^1 . \qedhere
\end{align*}
\end{proof}

\begin{lema}\label{NewLemmaEquivaTrans3} Let $x \in \Gamma(4)$ be an admissible element. If $S_x^4 = \llbracket x_1 \rrbracket  \cup \cdots \cup \llbracket x_r \rrbracket$ for some $x_1,\ldots,x_r\in {\rm Cl}(x)$, then $S_{x^2}^1 = [ x_1^{2} ] \cup \cdots \cup [ x_r^{2} ]$.
\end{lema}
\begin{proof}
Let $S_x^4 = \llbracket x_1 \rrbracket  \cup \cdots \cup \llbracket x_r \rrbracket$, where $x_1,\ldots,x_r\in {\rm Cl}(x)$. Then $S_{x^{-1}}^4 = \llbracket x_1^{-1} \rrbracket  \cup \cdots \cup \llbracket x_r ^{-1}\rrbracket$. Therefore
\begin{align*}
S_x^1 =S_x^4 \cup S_{x^{-1}}^4 = \left(\llbracket x_1 \rrbracket \cup \llbracket x_1^{-1} \rrbracket \right)  \cup \cdots \cup \left(\llbracket x_r \rrbracket \cup \llbracket x_r^{-1} \rrbracket\right)
=[ x_1 ]  \cup \cdots \cup [ x_r ].
\end{align*}
Now the result follows from Lemma~\ref{NewLemmaEquivaTrans2}.
\end{proof}

\noindent For $x \in \Gamma$ and $j\in \{1,\ldots,h\}$, define 
$$C_x(j):=\frac{1}{\chi_j({\mathbf 1})}\sum_{s \in S^1_x } \chi_{j}(s).$$ 
Note that $S_x^1 \in \mathbb{B}(\Gamma)$ and $C_x(j)$ is an eigenvalue of the normal simple Cayley graph ${\rm Cay}(\Gamma, S^1_x)$. As a consequence of Theorem~\ref{NorMixCayGraphInteg}, $C_x(j)$ is an integer for each $x \in \Gamma$ and $j\in \{1,\ldots,h\}$.

\begin{lema}\label{NewSum4mLemma11} Let $x\in \Gamma$ and ${\rm ord}(x)=2^{t}m$. If $m$ is odd and $t \geq 2$, then
$$C_x(j)=(\chi_j(x^m) + \chi_j(x^{3m})) C_{x^2}(j).$$ Moreover, $C_x(j)$ is an even integer for each $j\in \{1,\ldots,h\}$.
\end{lema}
\begin{proof} Let $S_x^1= [ x_1 ] \cup \cdots \cup [ x_k ]$ for some $x_1,\ldots,x_k\in {\rm Cl}(x)$. For $j\in \{1,\ldots,h\}$, we have
\begin{align}
C_x(j) &= \frac{1}{\chi_j({\mathbf 1})} \sum_{r=1}^{k}  \sum_{s \in [ x_r ] } \chi_j(s) \nonumber\\
&= \frac{1}{\chi_j({\mathbf 1})} \sum_{r=1}^{k} \bigg( \sum_{s \in [ x_r^2 ] } \chi_j(x_r^m) \chi_j(s) + \sum_{s \in [ x_r^2 ] } \chi_j(x_r^{3m}) \chi_j(s) \bigg) \nonumber\\
&=(\chi_j(x^m) + \chi_j(x^{3m})) \frac{1}{\chi_j({\mathbf 1})} \sum_{r=1}^{k}  \sum_{s \in [ x_r^2 ] }  \chi_j(s) \nonumber\\
&=(\chi_j(x^m) + \chi_j(x^{3m})) C_{x^2}(j). \label{EqNewSum4mLemma11}
\end{align} 
The second equality in the preceding equations follows from Part (iii) of Lemma~\ref{NewLemmaEquivaTrans1} and the fourth equality follows from Lemma~\ref{NewLemmaEquivaTrans2}.

We apply induction on $t$ to prove that $C_x(j)$ is an even integer.  Let $\rho_j$ be a representation corresponding to $\chi_j$. If $t=2$ then $\rho_j(x^m)^4$ is the identity matrix, and so each eigenvalue of $\rho_j(x^m)$ is a $4$-th root of unity. Thus, $\chi_j(x^m)$ is the trace of a matrix whose eigenvalues are $4$-th roots of unity. Therefore $\chi_j(x^m) + \chi_j(x^{3m})=2\Re(\chi_j(x^m))$, an even integer. Hence $C_x(j)$ is an even integer. Assume that the statement holds for each $z\in \Gamma$ with ${\rm ord}(z)=2^{t-1}m$, where $m$ is odd and $t-1 \geq 2$. Let $x\in \Gamma$ with ${\rm ord}(x)=2^{t}m$, where $m$ is odd and $t \geq 3$. Since the order of $x^2$ is $2^{t-1}m$, by induction hypothesis $C_{x^2}(j)$ is an even integer. If $C_{x^2}(j)= 0$, then clearly $C_{x}(j)= 0$, an even integer. By Equation~(\ref{EqNewSum4mLemma11}), $\chi_j(x^m) + \chi_j(x^{3m})$ is a rational algebraic integer whenever $C_{x^2}(j)\neq 0$. Thus if $C_{x^2}(j)\neq 0$, then $\chi_j(x^m) + \chi_j(x^{3m})$ is an integer. Hence by Equation~(\ref{EqNewSum4mLemma11}) and induction hypothesis, $C_x(j)$ is an even integer for each $j\in \{1,\ldots,h\}$. Thus the proof is complete by induction.
\end{proof}

Let $x \in \Gamma(4)$ be admissible. For $j\in \{1,\ldots,h\}$, define 
$$S_x(j):= \frac{{\mathbf i}}{\chi_j({\mathbf 1})} \sum_{s \in S^4_x} ( \chi_{j}(s)-\chi_{j}(s^{-1})).$$ 
Note that $S_x(j)$ is an H-eigenvalue of the normal oriented Cayley graph ${\rm Cay}(\Gamma, S^4_x)$ for each $j\in \{1,\ldots,h\}$. Since $S^4_x \in \mathbb{D}(\Gamma)$, by Theorem~\ref{NorOriCayGraphInteg} $S_x(j)$ is an integer for each $j\in \{1,\ldots,h\}$.

\begin{lema}\label{NewSum4mLemma33} Let $x\in \Gamma(4)$ be admissible and ${\rm ord}(x)=2^tm$. If $m$ is odd and $t\geq 2$, then
$$S_x(j)=\left\{ \begin{array}{ll}
					-2 \Im(\chi_j(x^{3m})) C_{x^2}(j) & \mbox{if } m \equiv 1 \Mod 4 \\
					-2 \Im(\chi_j(x^{m})) C_{x^2}(j) & \mbox{if } m \equiv 3 \Mod 4.
				\end{array}\right.$$ 
Moreover, $S_x(j)$ is an even integer for each $j\in \{1,\ldots,h\}$.
\end{lema}
\begin{proof} Let $S_x^4= \llbracket x_1 \rrbracket  \cup \cdots \cup \llbracket x_k \rrbracket$ for some $x_1,\ldots,x_k\in {\rm Cl}(x)$. For $j\in \{1,\ldots,h\}$, we have

\begin{equation*}
\begin{split}
S_x(j) &= \frac{{\mathbf i}}{\chi_j({\mathbf 1})} \sum_{r=1}^{k}  \sum_{s \in \llbracket x_r \rrbracket}( \chi_j(s)-\chi_j(s^{-1}))\\
&= \left\{ \begin{array}{ll}
					\frac{{\mathbf i}}{\chi_j({\mathbf 1})}\sum\limits_{r=1}^{k}  \sum\limits_{s \in \llbracket x_r \rrbracket}( \chi_j(s)-\chi_j(s^{-1})) & \mbox{if } m \equiv 1 \Mod 4 \\
					\frac{{\mathbf i}}{\chi_j({\mathbf 1})}\sum\limits_{r=1}^{k}  \sum\limits_{s \in \llbracket x_r \rrbracket}( \chi_j(s)-\chi_j(s^{-1})) & \mbox{if } m \equiv 3 \Mod 4
				\end{array}\right.\\
&= \left\{ \begin{array}{ll}
					\frac{{\mathbf i}}{\chi_j({\mathbf 1})}\sum\limits_{r=1}^{k}  \sum\limits_{s \in [ x_r^2 ] }( \chi_j(x_r^{3m}) \chi_j(s)-\chi_j(x_r^{-3m})\chi_j(s^{-1})) & \mbox{if } m \equiv 1 \Mod 4 \\
					\frac{{\mathbf i}}{\chi_j({\mathbf 1})}\sum\limits_{r=1}^{k}  \sum\limits_{s \in [ x_r^2 ]}( \chi_j(x_r^m) \chi_j(s)- \chi_j(x_r^{-m})\chi_j(s^{-1})) & \mbox{if } m \equiv 3 \Mod 4
				\end{array}\right.\\
&= \left\{ \begin{array}{ll}
					-2 \Im(\chi_j(x^{3m})) \frac{1}{\chi_j({\mathbf 1})}  \sum\limits_{r=1}^{k}  \sum\limits_{s \in [ x_r^2 ] }\chi_j(s) & \mbox{if } m \equiv 1 \Mod 4 \\
					-2 \Im(\chi_j(x^{m})) \frac{1}{\chi_j({\mathbf 1})}\sum\limits_{r=1}^{k}  \sum\limits_{s \in [ x_r^2 ]} \chi_j(s) & \mbox{if } m \equiv 3 \Mod 4
				\end{array}\right.\\
&= \left\{ \begin{array}{ll}
					-2 \Im(\chi_j(x^{3m})) C_{x^2}(j) & \mbox{if } m \equiv 1 \Mod 4 \\
					-2 \Im(\chi_j(x^{m})) C_{x^2}(j) & \mbox{if } m \equiv 3 \Mod 4.
				\end{array}\right.\\
\end{split} 
\end{equation*}
The third equality in the preceding equations follows from Part (i) of Lemma~\ref{NewLemmaEquivaTrans1} and the fifth equality follows from Lemma~\ref{NewLemmaEquivaTrans3}.
If $t=2$, then $\chi_j(x^{3m})$ and $\chi_j(x^{m})$ are traces of matrices whose eigenvalues are $4$-th roots of unity. Therefore, $\Im(\chi_j(x^{3m}))$ and $\Im(\chi_j(x^{m}))$ are integers. Thus $S_x(j)$ is an even integer. Now assume that $t \geq 3$. If $C_{x^2}(j) = 0$, then clearly $S_x(j)=0$, an even integer. Note that $2 \Im(\chi_j(x^{3m}))$ and $2 \Im(\chi_j(x^{m}))$ are rational algebraic integers whenever $C_{x^2}(j)\neq 0$. Thus if $C_{x^2}(j)\neq 0$, then $2 \Im(\chi_j(x^{3m}))$ and $2 \Im(\chi_j(x^{m}))$ are integers. Since the order of $x^2$ is $2^{t-1}m$, by Lemma~\ref{NewSum4mLemma11} $S_x(j)$ is an even integer.
\end{proof}

Let $S$ be a nonempty set in $\mathbb{D}(\Gamma)$ and that $S$ be expressible as a union of some conjugacy classes of $\Gamma$. Then $S$ is a skew-symmetric subset of $\Gamma$ that is closed under both conjugacy and the equivalence relation $\approx$. Let $S={\rm Cl}(x_1)\cup\cdots \cup  {\rm Cl}(x_k)=\llbracket y_1 \rrbracket  \cup \cdots \cup \llbracket y_r \rrbracket$ for some $x_1,\ldots,x_k,y_1,\ldots,y_r\in \Gamma(4)$. We see that
\begin{align*}
S=  {\rm Cl}(x_1)\cup\cdots \cup  {\rm Cl}(x_k)
= \left(\text{\footnotesize$\bigcup\limits_{s\in {\rm Cl}(x_1)}$} \llbracket s \rrbracket  \right)\cup \cdots \cup  \left(\text{\footnotesize$\bigcup\limits_{s\in {\rm Cl}(x_k)}$} \llbracket s \rrbracket  \right) 
= S_{x_1}^4 \cup\cdots \cup S_{x_k}^4 .
\end{align*}
Due to Lemma~\ref{partition2}, we can assume that the sets $S_{x_1}^4 ,\ldots , S_{x_k}^4$ are all distinct. In the next result, we prove the converse of Theorem~\ref{MinCharacEisensteinIntegCh4}.

\begin{theorem}\label{MinCharacEisensteinInteg2}
If $\Gamma$ is a finite group, then the normal mixed Cayley graph ${\rm Cay}(\Gamma,S)$ is Gaussian integral if and only if it is H-integral.
\end{theorem}
\begin{proof}
Assume that the normal mixed Cayley graph ${\rm Cay}(\Gamma,S)$ is H-integral. It is enough to show that $f_j(S)$ and $g_j(S)$ are integers for $j\in \{1,\ldots,h\}$. By Theorem~\ref{NormCayIntegChara}, we get  $S\setminus \overline{S}\in \mathbb{B}(\Gamma)$ and $\overline{S}\in \mathbb{D}(\Gamma)$. Therefore $\overline{S}=  S_{x_1}^{4} \cup\cdots\cup S_{x_r}^{4}$ for some $x_1,\ldots,x_r \in \Gamma(4)$, where the sets $S_{x_1}^{4}, \ldots , S_{x_r}^{4}$ are all distinct. By Lemma~\ref{NewSum4mLemma33}, $S_{x_1}(j)+\cdots+S_{x_r}(j)$ is an even integer. As $g_j(S)= \frac{1}{2}(S_{x_1}(j)+\cdots+S_{x_r}(j))$, we find that $g_j(S)$ is an integer.
Observe that $S_{x_i}^{4} \cup S_{x_i^{-1}}^{4} = S_{x_i}^{1}$, and so $\overline{S} \cup \overline{S}^{-1} = S_{x_1}^{1} \cup \cdots \cup S_{x_r}^{1}$. Therefore $f_j(S)=\frac{1}{\chi_{j}({\mathbf 1})}\sum\limits_{s \in S\setminus \overline{S}} \chi_j(s) + \frac{1}{2} (C_{x_1}(j) + \ldots + C_{x_r}(j))$. By Theorem~\ref{NorMixCayGraphInteg}, $\frac{1}{\chi_{j}({\mathbf 1})}\sum\limits_{s \in S\setminus \overline{S}} \chi_j(s)$ is an integer. Also, by Lemma~\ref{NewSum4mLemma11}, $C_{x_i}(j)$ is an even integer for each $i\in \{1,\ldots,r\}$. Hence we find that $f_j(S)$ is an integer. The other part of the theorem is already proved in Theorem~\ref{MinCharacEisensteinIntegCh4}.
\end{proof}

We give the following example to illustrate Theorem~\ref{MinCharacEisensteinInteg2}.

\begin{ex} \normalfont Consider the normal mixed Cayley graph ${\rm Cay}(M_{16}, S)$ of Example~\ref{example2}. We have already seen that it is H-integral, and hence it must be Gaussian integral. Indeed, using Theorem~\ref{EigNorColCayMix1}, the spectrum of  ${\rm Cay}(M_{16}, S)$ is obtained as  $$\{ [\gamma_{j}]^{1} \colon 1 \leq j \leq 8\}\cup \{[\gamma_{9}]^{4},[\gamma_{10}]^{4}\},\text{ where }$$   $$\gamma_j = \frac{1}{\chi_j(1)} \big[ \chi_j(a) + \chi_j(a^3) + \chi_j(a^5) + \chi_j(a^7) + \chi_j(a^3 x) + \chi_j(a^7 x) \big] \text{ for each } j\in \{1,\ldots,10\}.$$
We find that $\gamma_{1}=6$, $\gamma_{2}=-6$, $\gamma_{3}=2$, $\gamma_{4}=-2$, $\gamma_{6}=\gamma_{7}=2{\mathbf i}$, $\gamma_{5}=\gamma_{8}=-2{\mathbf i}$, and $\gamma_9=\gamma_{10}=0$. Thus ${\rm Cay}(M_{16}, S)$ is Gaussian integral. 
\end{ex}

\noindent\textbf{Conflict of Interest}: We declare that we have no conflict of interest to this work. 





\begin{thebibliography}{10}



\bibitem{ahmadi2009graphs}
O.~Ahmadi, N.~Alon, I.F. Blake, and I.E. Shparlinski.
\newblock Graphs with integral spectrum.
\newblock {\em Linear Algebra and its Applications} 430(1) (2009), 547--552.



\bibitem{alperin2012integral}
R.C. Alperin and B.L. Peterson.
\newblock Integral sets and Cayley graphs of finite groups.
\newblock {\em The Electronic Journal of Combinatorics} 19(1) (2012), \#P44.

\bibitem{balinska2002survey}
K.~Bali{\'n}ska, D.~Cvetkovi{\'c}, Z.~Radosavljevi{\'c}, S.~Simi{\'c}, and
  D.~Stevanovi{\'c}.
\newblock A survey on integral graphs.
\newblock  {\em Univ. Beograd, Publ. Elektrotehn. Fak. Ser. Mat} 13 (2003), 42--65.

\bibitem{bapat2012weighted}
R.B.~Bapat, D.~Kalita, and S.~Pati.
\newblock On weighted directed graphs.
\newblock {\em Linear Algebra and its Applications}, 436(1) (2012), 99--111.

\bibitem{bridges1982rational}
W.G. Bridges and R.A. Mena.
\newblock Rational g-matrices with rational eigenvalues.
\newblock {\em Journal of Combinatorial Theory, Series A} 32(2) (1982), 264--280.

%
%

\bibitem{cheng2019integral}
T.~Cheng, L.~Feng, and H.~Huang.
\newblock Integral Cayley graphs over dicyclic group.
\newblock {\em Linear Algebra and its Applications} 566 (2019), 121--137.

\bibitem{csikvari2010integral}
P.~Csikv{\'a}ri.
\newblock Integral trees of arbitrarily large diameters.
\newblock {\em Journal of Algebraic Combinatorics} 32(3) (2010), 371--377.

\bibitem{dummitandfoote}
D.S.~Dummit and R.M.~Foote.
\newblock {\em Abstract Algebra, Third Edition}. Hoboken: Wiley, 2004.

\bibitem{foster2016spectra}
B.~Foster-Greenwood and C.~Kriloff.
\newblock Spectra of Cayley graphs of complex reflection groups.
\newblock {\em Journal of Algebraic Combinatorics} 44(1) (2016), 33--57.

\bibitem{godsil2014rationality}
C.~Godsil and P.~Spiga.
\newblock Rationality conditions for the eigenvalues of normal finite cayley graphs.
\newblock { arXiv preprint arXiv:1402.5494}, 2014.

\bibitem{2017mixed}
K.~Guo and B.~Mohar.
\newblock Hermitian adjacency matrix of digraphs and mixed graphs.
\newblock {\em Journal of Graph Theory} 85(1) (2017), 217--248.

\bibitem{harary1974graphs}
F.~Harary and A.J. Schwenk.
\newblock Which graphs have integral spectra?
\newblock Lecture Notes in Mathematics 406, Springer Verlag, 1974, 45--50.



\bibitem{kadyanHS-IntegSecKind}
M.~Kadyan.
\newblock HS-integral and Eisenstein integral normal mixed Cayley graphs. arXiv:2201.08160.




\bibitem{kadyan2021integral}
M.~Kadyan and B.~Bhattacharjya.
\newblock Integral mixed circulant graph.
\newblock {\em Discrete Mathematics}, 346(1):113142, 2022.

\bibitem{kadyan2021integralAbelian}
M. Kadyan and B. Bhattacharjya.
\newblock Integral mixed Cayley graphs over abelian group.
\newblock {\em The Electronic Journal of Combinatorics} 28(4) (2021), \#P4.46.

\bibitem{klotz2010integral}
W.~Klotz and T.~Sander.
\newblock Integral Cayley graphs over abelian groups.
\newblock {\em The Electronic Journal of Combinatorics} 17 (2010), \#R81.

\bibitem{ku2015cayley}
C.Y. Ku, T.~Lau, and K.B. Wong.
\newblock Cayley graph on symmetric group generated by elements fixing k
  points.
\newblock {\em Linear Algebra and its Applications} 471 (2015), 405--426.


\bibitem{li2013circulant}
F.~Li.
\newblock Circulant digraphs integral over number fields.
\newblock {\em Discrete Mathematics} 313(6) (2013), 821--823.

\bibitem{2015mixed}
J.~Liu and X.~Li.
\newblock Hermitian-adjacency matrices and hermitian energies of mixed graphs.
\newblock {\em Linear Algebra and its Applications} 466 (2015), 182--207.

\bibitem{lu2018integral}
L.~Lu, Q.~Huang, and X.~Huang.
\newblock Integral Cayley graphs over dihedral groups.
\newblock {\em Journal of Algebraic Combinatorics} 47(4) (2018), 585--601.

\bibitem{2006integral}
W.~So.
\newblock Integral circulant graphs.
\newblock {\em Discrete Mathematics} 306(1) (2006), 153--158.

\bibitem{steinberg2009representation}
B.~Steinberg.
\newblock {\em Representation theory of finite groups}.
\newblock Springer New York, 2009.

%
%

\bibitem{watanabe1979note}
M.~Watanabe.
\newblock Note on integral trees.
\newblock {\em Mathematics Reports} 2 (1979), 95--100.

\bibitem{watanabe1979integral}
M.~Watanabe and A.J. Schwenk.
\newblock Integral starlike trees.
\newblock {\em Journal of the Australian Mathematical Society} 28(1) (1979), 120--128.

\bibitem{xu2011gaussian}
Y.~Xu and J.~Meng.
\newblock Gaussian integral circulant digraphs.
\newblock {\em Discrete Mathematics} 311(1) (2011), 45--50.

\end{thebibliography}
\end{document}